\documentclass[preprint]{elsarticle}
\usepackage{amsmath,amsthm,latexsym,amsfonts,amssymb,mathrsfs,array,manfnt}
\usepackage[all]{xy}
\usepackage{paralist}
\usepackage{framed}
\usepackage{cancel}
\usepackage{enumerate}
\usepackage{units}
\usepackage{gensymb}
\usepackage{setspace}
\usepackage{amscd}
\usepackage{microtype}
\usepackage{tikz}
\usepackage{hyperref}
\usepackage{caption}
\usetikzlibrary{matrix}

\newcommand{\newword}[1]{\textbf{\emph{#1}}}

\newcommand{\into}{\hookrightarrow}

\newcommand{\abs}[1]{\left\lvert#1\right\rvert}
\newcommand{\norm}[1]{\left\lVert#1\right\rVert}

\newcommand{\Z}{\ensuremath{\mathbb{Z}}}
\newcommand{\R}{\ensuremath{\mathbb{R}}}
\newcommand{\C}{\ensuremath{\mathbb{C}}}
\newcommand{\Q}{\ensuremath{\mathbb{Q}}}

\newcommand{\M}{\mathcal{M}}

\renewcommand{\P}{\ensuremath{\mathbb{P}}}
\renewcommand{\H}{\ensuremath{\mathcal{H}}}

\renewcommand{\bar}[1]{\overline{#1}}

\DeclareMathOperator{\Hilb}{Hilb}

\DeclareMathOperator{\diag}{diag}

\DeclareMathOperator{\Pic}{Pic}

\DeclareMathOperator{\Id}{Id}

\DeclareMathOperator{\Aut}{Aut}

\DeclareMathOperator{\im}{im}

\DeclareMathOperator{\Sym}{Sym}

\DeclareMathOperator{\pr}{pr}

\usepackage{color}

\newcommand{\cF}{\mathcal{F}}

\newcommand{\cH}{\mathcal{H}}

\newcommand{\cJ}{\mathcal{J}}

\newcommand{\cL}{\mathcal{L}}
\newcommand{\cM}{\mathcal{M}}

\newcommand{\cR}{\mathcal{R}}

\newcommand{\cU}{\mathcal{U}}
\newcommand{\cV}{\mathcal{V}}

\newcommand{\bA}{\mathbf{A}}
\newcommand{\bB}{\mathbf{B}}

\newcommand{\bP}{\mathbf{P}}

\theoremstyle{plain}
\newtheorem{theorem}{Theorem}
\numberwithin{theorem}{section}

\newtheorem{thm}[theorem]{Theorem}

\newtheorem{prop}[theorem]{Proposition}

\newtheorem{cor}[theorem]{Corollary}

\newtheorem{lem}[theorem]{Lemma}

\theoremstyle{definition}
\newtheorem{Definition/Theorem}[theorem]{Definition/Theorem}
\newtheorem{Definition/Proposition}[theorem]{Definition/Proposition}

\newtheorem{Def}[theorem]{Definition}
\newtheorem{Defthm}[theorem]{Definition-Theorem}
\newtheorem{ex}[theorem]{Example}

\newtheorem{Corollary/Definition}[theorem]{Corollary/Definition}

\theoremstyle{remark}

\newtheorem{rem}[theorem]{Remark}

\usepackage{fullpage}
\usepackage{longtable}
\renewcommand{\H}{\cH}
\renewcommand{\M}{\cM}
\newcommand{\Hbar}{\bar{\cH}}
\newcommand{\Mbar}{\bar{\cM}}
\newcommand{\Ubar}{\bar{\cU}}
\newcommand{\Hfull}{\cH^{\mathrm{full}}}
\newcommand{\Hbarfull}{\bar{\cH}^{\mathrm{full}}}
\newcommand{\Afull}{\bA^{\mathrm{full}}}

\renewcommand{\rm}{\mathrm{rm}}
\newcommand{\cv}{\mathrm{cv}}
\renewcommand{\setminus}{\smallsetminus}
\renewcommand{\vert}{\mathrm{vert}}
\newcommand{\be}{\boldsymbol{\epsilon}}

\DeclareMathOperator{\Thurst}{Thurst}

\usetikzlibrary{arrows}
\usetikzlibrary{calc}
\newcommand{\rat}{
\tikz[minimum height=2ex]
  \path[dashed,->]
   node (a)            {}
   node (b) at (1em,0) {}
  ($(a.center)+(0,-.07)$) edge ($(b.center)+(0,-.07)$);
}
\newcommand{\ratt}{
\tikz[minimum height=2ex]
  \path[dashed,->]
   node (a)            {}
   node (b) at (1em,0) {}
  ($(a.center)+(0,.0)$) edge ($(b.center)+(0,.0)$)
  ($(a.center)+(0,-.14)$) edge ($(b.center)+(0,-.14)$);
}

\DeclareMathOperator{\md}{md}
\DeclareMathOperator{\br}{br}

\DeclareMathOperator{\Gr}{Gr}
\DeclareMathOperator{\full}{full}

\definecolor{darkgreen}{rgb}{0,0.7,0}

\begin{document}
\begin{frontmatter}
  \title{Hurwitz correspondences on compactifications of $\M_{0,N}.$}
  \author{Rohini Ramadas}
  \ead{ramadas@umich.edu}
  \address{Department
    of Mathematics, University of Michigan, Ann Arbor, MI 48109}
  \date{\today}
  \begin{abstract}
    Hurwitz correspondences are certain multivalued self-maps of the moduli space $\M_{0,N}$. They arise in the study of Thurston's topological characterization of rational functions. We consider the dynamics of Hurwitz correspondences and ask: On which compactifications of $\M_{0,N}$ should they be studied? We compare a Hurwitz correspondence $\H$ across various modular compactifications of $\M_{0,N}$, and find a weighted stable curves compactification $X_N^\dagger$ that is optimal for its dynamics. We use $X_N^\dagger$ to show that the $k$th dynamical degree of $\H$ is the absolute value of the dominant eigenvalue of the pushforward induced by $\H$ on a natural quotient of $H_{2k}(\Mbar_{0,N})$.
  \end{abstract}
\end{frontmatter}

\section{Introduction}
We recall the moduli space $\M_{0,\bP}$, a quasiprojective variety
parametrizing smooth genus zero curves with marked points labeled by
the elements of a finite set $\bP.$ Let $\H$ be a \emph{Hurwitz space}
parametrizing maps, with prescribed branching, from one $\bP$-marked
genus zero curve to another. $\H$ admits two maps to $\M_{0,\bP}$ ---
a `target curve' map $\pi_1$, and a `source curve' map $\pi_2$. If all
the branch values of the target curve are marked, $\pi_1$ is a
covering map, so $\pi_2\circ\pi_1^{-1}$ is a multivalued map --- a
\emph{Hurwitz correspondence} --- from $\M_{0,\bP}$ to itself. We
study the dynamics of Hurwitz correspondences via invariants called
\emph{dynamical degrees}.

Dynamical degrees are numerical invariants associated to a self-map of
a smooth projective variety
(\cite{Friedland1991,RussakovskiiShiffman1997}). They measure
dynamical complexity of iteration. For $g:X\to X$ a surjective regular
map, the $k$th dynamical degree of $g$ is the absolute value of the
dominant eigenvalue of the pullback $g^*:H^{k,k}(X)\to H^{k,k}(X).$
For $g$ a rational self-map or a rational multivalued self-map, we may
not have $(g^n)^*=(g^*)^n$. In these cases the $k$th dynamical degree
is defined to be
\begin{align*}
  \lim_{n\to\infty}\norm{(g^n)^*:H^{k,k}(X)\to H^{k,k}(X)}^{1/n}.
\end{align*}
Suppose we do have $(g^n)^*=(g^*)^n$ on $H^{k,k}(X)$ --- $g$ may not
be regular, but it behaves like a regular map in terms of its action
on $H^{k,k}(X).$ Then $g$ is called $k$-stable, and the $k$th
dynamical degree of $g$ is the absolute value of the dominant
eigenvalue of $g^*:H^{k,k}(X)\to H^{k,k}(X)$. Computing or
characterizing dynamical degrees of a map $g$ which is \emph{not}
$k$-stable involves dealing with the pullbacks along infinitely many
iterates $g^n$, and is notoriously difficult
(\cite{Roeder2013,Bedford2011}). The dynamical degrees of a map
provide information about a more fundamental invariant --- topological
entropy. The topological entropy of a regular map is the logarithm of
the largest of its dynamical degrees (\cite{Gromov2003}); the
topological entropy of a rational map or rational multivalued map is
bounded from above by the logarithm of the largest of its dynamical
degrees (\cite{DinhSibony2005,DinhSibony2008}).

$\M_{0,\bP}$ is not compact --- to make sense of the dynamical degrees
of a Hurwitz correspondence $\H$ we must consider it as a
\emph{rational} multivalued self-map of some projective
compactification $X_{\bP}$ of $\M_{0,\bP}.$ We are free to choose this
compactification $X_{\bP}$ for convenience --- the dynamical degrees
of $\H$ are birational invariants
(\cite{DinhSibony2005,DinhSibony2008,Truong2015,Ramadas2016}) that
purely reflect its dynamics on the interior, $\M_{0,\bP}$. Koch and
Roeder proved in \cite{KochRoeder2015} that some special examples are
$k$-stable on the stable curves compactification $\Mbar_{0,\bP}$, and
computed their dynamical degrees. In this work, we exploit the moduli
space interpretation of Hurwitz correspondences to provide an account
of their homological actions on various compactifications of
$\M_{0,\bP}$. Together with a companion paper \cite{Ramadas2016}, this
comprises the first systematic study of the dynamical degrees of any
natural family of rational multivalued maps.

We first prove, jointly with Sarah Koch and David Speyer:

\medskip

\noindent\textbf{Proposition \ref{prop:KComposable}} (Koch, Ramadas,
Speyer)\textbf{.}  \textit{Let $\H$ be a Hurwitz correspondence on
  $\M_{0,\bP}.$ Then for all $k$, $\H$ is $k$-stable as a rational
  multivalued self-map of $\Mbar_{0,\bP}$.}

\medskip

Since $H^{k,k}(\Mbar_{0,\bP})=H^{2k}(\Mbar_{0,\bP}),$ and since
pushforward and pullback maps are dual to each other (Section
\ref{sec:RationalCorrespondences}), we obtain:

\medskip

\noindent \textbf{Corollary \ref{cor:LargestEigenvalue}} (Koch,
Ramadas, Speyer)\textbf{.} \textit{The $k$th dynamical degree of $\H$ is the
absolute value of the dominant eigenvalue of the induced pushforward
  \begin{align*}
    [\H]_*:H_{2k}(\Mbar_{0,\bP})\to H_{2k}(\Mbar_{0,\bP}).
  \end{align*}}

In fact, $[\H]_*$ preserves the cone of effective classes in
$H_{2k}(\Mbar_{0,\bP})$, so it follows from the theory of
cone-preserving operators that $[\H]_*$ has a nonnegative dominant
eigenvalue, with a \emph{pseudoeffective} eigenvector (Remark
\ref{rem:EffectiveCone}).

We use Harris and Mumford's compactification of $\H$ by a moduli space
of \emph{admissible covers} (\cite{HarrisMumford1982}) to show that
the multivalued map $\pi_2\circ\pi_1^{-1}$, considered as a map from
$\M_{0,\bP}$ to $\Sym^d(\M_{0,\bP})$, extends to a \emph{regular} map
from $\Mbar_{0,\bP}$ to $\Sym^d(\Mbar_{0,\bP}).$ Thus, Hurwitz
correspondences extend to $\Mbar_{0,\bP}$ ``without indeterminacy.''
Proposition \ref{prop:KComposable} follows as a consequence.

We next introduce a filtration
$\{\Lambda_{\bP}^{\le\lambda}\}_\lambda$ of $H_{2k}(\Mbar_{0,\bP})$
indexed by the partially ordered set $\{\mbox{partitions $\lambda$ of
  $k$}\}$ (Section \ref{sec:Filtration}). We show:

\medskip

\noindent\textbf{Theorem \ref{thm:PreservesFiltration}.} \textit{Let $\H$ be
any Hurwitz correspondence on $\M_{0,\bP}$. Then
$[\H]_*:H_{2k}(\Mbar_{0,\bP})\to H_{2k}(\Mbar_{0,\bP})$ sends each
subspace $\Lambda_{\bP}^{\le\lambda}$ to itself.}

\medskip

Thus the operator $[\H]_*$ can be written, in multiple different ways,
as a block-lower-triangular matrix. The blocks give the induced action
of $[\H]_*$ on successive quotients of subspaces of the filtration
$\{\Lambda_{\bP}^{\le\lambda}\}.$

Which of these blocks contains the dominant eigenvalue --- the
dynamical degree? Set
$$\Lambda_{\bP}^{<(k)}:=\sum_{\Lambda_{\bP}^{\le\lambda}\subsetneq
  H_{2k}(\Mbar_{0,\bP})}\Lambda_{\bP}^{\le\lambda}$$
and
$$\Omega_{\bP}^k:=H_{2k}(\Mbar_{0,\bP})/\Lambda_{\bP}^{<(k)}.$$
Then $\Omega_{\bP}^k$ is the smallest possible nonzero quotient of
$H_{2k}(\Mbar_{0,\bP})$ obtainable by subspaces in the filtration
$\{\Lambda_{\bP}^{\le\lambda}\}.$ The induced action
$[\H]_*:\Omega_{\bP}^k\to\Omega_{\bP}^k$ gives the topmost block of
the block-lower-triangular matrix
$[\H]_*:H_{2k}(\Mbar_{0,\bP})\to H_{2k}(\Mbar_{0,\bP}).$ We show:

\medskip

\noindent\textbf{Theorem \ref{thm:DynamicalDegreeInTopBlock}.}
\textit{The $k$th dynamical degree of $\H$ is the absolute value of
  the dominant eigenvalue of
  $[\H]_*:\Omega_{\bP}^k\to\Omega_{\bP}^k.$}

\medskip

This considerably increases the efficiency of computing dynamical
degrees; for example, when $k=\dim\M_{0,\bP}-1,$ the dimension of
$H_{2k}(\Mbar_{0,\bP})$ is
$\frac{2^{\abs{\bP}}-\abs{\bP}^2+\abs{\bP}-2}{2},$ while the dimension
of $\Omega_{\bP}^k$ is only $\abs{\bP}.$

We prove Theorem \ref{thm:DynamicalDegreeInTopBlock} by finding an
alternate compactification $X_{\bP}^{\dagger}$ of $\M_{0,\bP}$ (a
Hassett moduli space of \emph{weighted} stable curves) with a
birational morphism $\rho:\Mbar_{0,\bP}\to X_{\bP}^{\dagger}$ that
contracts the cycles in $\Lambda_{\bP}^{<(k)}.$ Interestingly,
for $k\ge\frac{\dim\M_{0,N}}{2},$ the \emph{only} $k$-cycles
contracted by $\rho$ are those in $\Lambda_{\bP}^{<(k)}$, and although
Hurwitz correspondences do not extend to $X_{\bP}^\dagger$ without
indeterminacy, they are $k$-stable on $X_{\bP}^\dagger.$

$\Mbar_{0,\bP}$ has ``enough boundary'' to resolve any potential
indeterminacy of $\H$. However, Theorems \ref{thm:PreservesFiltration}
and \ref{thm:DynamicalDegreeInTopBlock} together imply that the
boundary $\Mbar_{0,\bP}\setminus\M_{0,\bP}$, and the homology groups
of $\Mbar_{0,\bP}$, are too large. The action of $[\H]_*$ on
$\Omega_{\bP}^k$ is ``intrinsic'' to the dynamics of $\H$ on
$\M_{0,\bP}$ --- this action determines the dynamical degree, and for
$k\ge\frac{\dim\M_{0,N}}{2}$ is isomorphic to the $k$-stable
homological action of $\H$ on the birational model $X_{\bP}^\dagger.$
In contrast, the action on the subspace $\Lambda_{\bP}^{<(k)}$ is an
artifact of the compactification $\Mbar_{0,\bP}$ --- this action is
``contained in the boundary'' (Corollary
\ref{cor:ContainedInBoundary}). We note that Theorems
\ref{thm:PreservesFiltration} and \ref{thm:DynamicalDegreeInTopBlock}
do not help give an upper bound for the entropy of $\H$. In
(\cite{Ramadas2016}), we show that for any Hurwitz correspondence
$\H$, the sequence $k\mapsto(\mbox{$k$th dynamical degree of $\H$})$
is nonincreasing. Thus the largest dynamical degree --- the one
providing an upper bound for entropy --- is the 0th. This is the
topological degree of the `target curve' map $\pi_1:\H\to\M_{0,\bP};$
a \emph{Hurwitz number}.

Our study is partly motivated by a connection to Teichm\"uller
theory. Hurwitz correspondences arise in the context of a criterion
given by W. Thurston (\cite{DouadyHubbard1993}) for the existence of
an algebraic self-map of $\P^1$ of a given topological type. Let
$\phi:S^2\to S^2$ be an orientation-preserving branched covering from
a topological 2-sphere to itself. When is $\phi$ conjugate, up to
homotopy, to a rational function $f:\P^1\to\P^1$? Suppose $\phi$ has
finite \emph{post-critical set}
$$\bP:=\{\phi^n(x)|n>0,\mbox{$x$ a critical point of $\phi$\}}.$$
Denote by $\mathcal{T}(S^2,\bP)$ the \emph{Teichm\"uller space}
parametrizing complex structures on the marked sphere
$(S^2,\bP)$. $\mathcal{T}(S^2,\bP)$ is a nonalgebraic complex
manifold, and the universal cover of $\M_{0,\bP}$. Given one complex
structure on $(S^2,\bP)$, one may pull it back along the branched
covering $\phi$ to obtain another. This defines a holomorphic self-map
of $\mathcal{T}(S^2,\bP)$ called the Thurston pullback map (Section
\ref{sec:Teichmuller}). A fixed point of the Thurston pullback map is
a complex structure under which $\phi$ is identified with a rational
function $f:\P^1\to\P^1.$ Although the Thurston pullback map does not
descend to a self-map of $\M_{0,\bP}$, it does descend to a
multivalued self-map --- in fact, a Hurwitz correspondence (Koch,
\cite{Koch2013}). The dynamics of Hurwitz correspondences thus
algebraically record the dynamics of the nonalgebraic Thurston
pullback map.

\medskip

\noindent\textbf{Outline.} Sections
\ref{sec:RatCorDynDeg} through \ref{sec:AdmissibleCovers} give
background on, in order, rational correspondences, Hurwitz
correspondences, compactifications of $\M_{0,N}$, and the admissible
covers compactifications of Hurwitz spaces. Our results are in
Sections \ref{sec:HurwitzStableCurves}, \ref{sec:Filtration}, and
\ref{sec:KComposability}. In Section \ref{sec:HurwitzStableCurves}, we
show that Hurwitz correspondences are $k$-stable on $\Mbar_{0,N}$.  In
Section \ref{sec:Filtration}, we show they preserve a natural
filtration of $H_{2k}(\Mbar_{0,N}).$ In Section
\ref{sec:KComposability}, we use alternate compactifications of
$\M_{0,N}$ to investigate where in this filtration the dynamical
degree lies.

\section{Acknowledgments}
I am grateful to my advisor David Speyer and to my co-advisor Sarah
Koch for introducing me to Hurwitz correspondences and for excellent
guidance along the way. Proposition \ref{prop:KComposable} and
Corollary \ref{cor:LargestEigenvalue} are joint with them. I am
grateful to Rob Silversmith, Karen Smith, Anand Deopurkar, Anand
Patel, and Renzo Cavalieri for useful conversations, to Felipe
P\'erez, Paul Reschke, Mattias Jonsson, and Jake Levinson for helpful
comments on earlier drafts, and to Rob Silversmith for help with
typing and figures. 

Funding: This research was conducted at the University of Michigan,
partially supported by NSF grants 0943832, 1045119, and 1068190.

\section{Conventions}
All varieties and schemes are over $\C$. For $X$ a smooth projective
variety, we denote by $H_c(X)$ and $H^c(X)$ its $c$th singular
homology and cohomology groups respectively in the analytic topology
with coefficients in $\R$. There are ``Poincar\'e duality''
isomorphisms between $H^c(X)$ and $H_{\dim_\R X-c}(X).$ These allow us
to define pushforward maps on cohomology groups and pullback maps on
homology groups, and cup product of homology classes. For $X$ any
variety, we denote by $A_k(X)$ the Chow group of $k$-cycles on $X$ up
to rational equivalence.

All curves in this paper are projective reduced curves, connected but
not necessarily irreducible, of arithmetic genus zero.

A partition $\lambda$ of a positive integer $k$ is a multiset of
positive integers whose sum with multiplicity is $k$. For example, we
write $(1,1,2)$ for the partition $4=1+1+2$. If $\lambda(j)$ is a
partition of $k(j),$ then we denote by $\cup_j\lambda(j)$ the multiset
union, which is a partition of $\sum_jk(j).$ A multiset $\lambda_1$ is
a \newword{submultiset} of $\lambda_2$ if for all $r\in\lambda_1,$ the
multiplicity of occurrence of $r$ in $\lambda_1$ is less than or equal
to the multiplicity of occurrence of $r$ in $\lambda_2$. A set
partition of a finite set $\bA$ is a set of nonempty subsets of $\bA,$
with pairwise empty intersection, whose union is $\bA.$

Let $\Lambda$ be a set with some algebraic structure, e.g. a vector
space. A \newword{poset-filtration} of $\Lambda$ is a collection
$\Lambda^\lambda$ of sub-objects of $\Lambda$ indexed by elements of a
partially ordered set $(\{\lambda\},\le)$ with the property that
$\Lambda^{\lambda_1}\subseteq\Lambda^{\lambda_2}$ whenever
$\lambda_1\le\lambda_2.$

\section{Rational correspondences and dynamical degrees}\label{sec:RatCorDynDeg}
In this section $X$, $X_1,$ $X_2,$ and $X_3$ are smooth irreducible
projective varieties.
\subsection{Rational
  correspondences}\label{sec:RationalCorrespondences}
A rational correspondence from $X_1$ to $X_2$ is a multivalued map
from a dense open set in $X_1$ to $X_2.$
\begin{Def}
  A \newword{rational correspondence}
  $(\Gamma,\pi_1,\pi_2):X_1\ratt X_2$ is a diagram
  \begin{center}
    \begin{tikzpicture}
      \matrix(m)[matrix of math nodes,row sep=3em,column
      sep=4em,minimum width=2em] {
        &\Gamma&\\
        X_1&&X_2\\}; \path[-stealth] (m-1-2) edge node [above left]
      {$\pi_1$} (m-2-1); \path[-stealth] (m-1-2) edge node
      [above right]
      {$\pi_2$} (m-2-3);
    \end{tikzpicture}
  \end{center}
  where $\Gamma$ is a smooth quasiprojective variety, not necessarily
  irreducible, and the restriction of $\pi_1$ to every irreducible
  component of $\Gamma$ is dominant and generically finite.
\end{Def}
We sometimes suppress part of the notation and write $\Gamma:X_1\ratt
X_2$ for the rational correspondence $(\Gamma,\pi_1,\pi_2):X_1\ratt
X_2$.

Over a dense open subset $U_1\subseteq X_1,$ $\pi_1$ is a finite
covering map of some degree $d$, so $\pi_2\circ\pi_1^{-1}$ defines a
multivalued map from $U_1$ to $X_2$ and induces a regular map from
$U_1$ to $\Sym^d(X_2)$. Outside $U_1,$ however, the fibers of $\pi_1$
could be empty or positive-dimensional, and it may be impossible to
extend this to a multivalued map from $X_1$ to $X_2$, respectively a
regular map from $X_1$ to $\Sym^d(X_2)$.

Rational correspondences induce pushforward and pullback maps:
\begin{Def}\label{Def:Cycle}
  Let $\bar{\Gamma}$ be a smooth projective compactification of
  $\Gamma$ such that $\Gamma$ is dense in $\bar{\Gamma}$ and $\pi_1$
  and $\pi_2$ extend to maps $\bar{\pi_1}$ and $\bar{\pi_2}$ defined
  on $\bar{\Gamma}$. The cycle
  $(\bar{\pi_1}\times\bar{\pi_2})_*([\bar{\Gamma}])\in H_{2\dim
    X_1}(X_1\times X_2)$
  is independent of the choice of compactification $\bar{\Gamma}$, so
  we denote this cycle by $[\Gamma]$.
Set
\begin{align*}
  [\Gamma]_*:&=(\bar{\pi_2})_*\circ\bar{\pi_1}^*:H_c(X_1)\to H_c(X_2)
\end{align*}
and
\begin{align*}
  [\Gamma]^*:&=(\bar{\pi_1})_*\circ\bar{\pi_2}^*:H^c(X_2)\to H^c(X_1).
\end{align*}
\end{Def}
$[\Gamma]_*$ and $[\Gamma]^*$ are well-defined and depend only on
$[\Gamma]$. In fact, if $\pr_1$ and $\pr_2$ are the projections from
$X_1\times X_2$ to $X_1$ and $X_2$ respectively, we have for
$\mathfrak{a}\in H_c(X_1)$ and $\mathfrak{b}\in H^c(X_2)$
\begin{align*}
  [\Gamma]_*(\mathfrak{a})=(\pr_2)_*([\Gamma]\smile\pr_1^*(\mathfrak{a}))
\end{align*}
and
\begin{align*}
  [\Gamma]^*(\mathfrak{b})=(\pr_1)_*([\Gamma]\smile\pr_2^*(\mathfrak{b})),
\end{align*}
where $\smile$ denotes cup product in $H^*(X_1\times X_2).$ The
cohomology group $H^c(X_j)$ is dual to the homology group
$H_c(X_j)$. By the projection formula, $[\Gamma]_*$ and $[\Gamma]^*$
are dual maps.

Suppose $(\Gamma,\pi_1,\pi_2):X_1\ratt X_2$ and
$(\Gamma',\pi_2',\pi_3'):X_2\ratt X_3$ are rational correspondences
such that the image under $\pi_2$ of every irreducible component of
$\Gamma$ intersects the domain of definition of the multivalued
function $\pi_3'\circ(\pi_2')^{-1}.$ The composite
$\Gamma'\circ\Gamma$ is a rational correspondence from $X_1$ to $X_3$
defined as follows.

Pick dense open sets $U_1\subseteq X_1$ and $U_2\subseteq X_2$ such that
$\pi_2(\pi_1^{-1}(U_1))\subseteq U_2,$ and
$\pi_1|_{\pi_1^{-1}(U_1)}$ and
$\pi_2'|_{(\pi_2')^{-1}(U_2)}$ are both covering maps. Set
$$\Gamma'\circ
\Gamma:=\pi_1^{-1}(U_1)\thickspace\thickspace{_{\pi_2}\times_{\pi_2'}}\thickspace\thickspace(\pi_2')^{-1}(U_2),$$
together with its given maps to $X_1$ and $X_3$.

Although this definition of $\Gamma'\circ\Gamma$ depends on the choice
of open sets $U_1$ and $U_2,$ the cycle class $[\Gamma'\circ\Gamma]$
is well-defined; in fact, if $[\Gamma_1]=[\Gamma_2]$ and
$[\Gamma_1']=[\Gamma_2'],$ then
$[\Gamma_1'\circ\Gamma_1]=[\Gamma_2'\circ\Gamma_2]$
(\cite{DinhSibony2008}). It will be convenient to work with a concrete
representative $\Gamma'\circ\Gamma$ in its cycle class. However, composition of
rational correspondences is only defined up to equivalence of cycle
class.

Unfortunately, it is not necessarily true that
$[\Gamma'\circ\Gamma]_*=[\Gamma']_*\circ[\Gamma]_*$ or that
$[\Gamma'\circ\Gamma]^*=[\Gamma]^*\circ[\Gamma']^*$. (See Section
\ref{sec:Cremona} for an example.) Suppose
$$[\Gamma'\circ\Gamma]_*=[\Gamma']_*\circ[\Gamma]_*:H_c(X_1)\to
H_c(X_3).$$ Then we say $\Gamma'$ and $\Gamma$ are
\newword{$c$-homologically composable}. In this case, by duality of
pushforward and pullback, $\Gamma'$ and $\Gamma$ are also
$c$-cohomologically composable:
$$[\Gamma'\circ\Gamma]^*=[\Gamma]^*\circ[\Gamma']^*:H^c(X_3)\to
H^c(X_1).$$

\begin{rem}
  There is a theory of \newword{correspondences}, as distinct from
  rational correspondences. A correspondence from $X_1$ to $X_2$ is a
  cycle class in $X_1\times X_2.$ Correspondences also induce maps on
  (co)homology groups, but these maps are functorial under
  composition (\cite{Fulton1998}).
\end{rem}

\subsection{Any rational map is a rational correspondence}
Let $g:X_1\rat X_2$ be a rational map. Then
$\Gr(g):=\overline{\{(x,g(x))\}}\subseteq X_1\times X_2$ together with
its projections $\pr_1$ and $\pr_2$ to $X_1$ and $X_2$ respectively,
is a rational correspondence from $X_1$ to $X_2.$ The pushforward
$g_*$ and the pullback $g^*$ by $g$ have been independently defined in
the literature to be $[\Gr(g)]_*$ and $[\Gr(g)]^*$, respectively
(\cite{Roeder2013}). If $g:X_1\rat X_2$ and $g':X_2\rat X_3$ are
rational maps such that the image of $g$ intersects the domain of
definition of $g',$ then the composite $g'\circ g$ is a rational map
$X_1\rat X_3$, and $[\Gr(g'\circ g)]=[\Gr(g')\circ\Gr(g)].$ We may
thus identify rational maps with the rational correspondences given by
their graphs.

If $(\Gamma,\pi_1,\pi_2):X_1\ratt X_2$ is a rational correspondence
with $\pi_1$ generically one-to-one, then $[\Gamma]$ is the
rational correspondence given by the rational map
$\pi_2\circ\pi_1^{-1}.$

\subsection{Dynamical degrees}\label{sec:DynamicalDegrees}
We refer the reader to \cite{Roeder2013} or \cite{Bedford2011} for
more extended discussions of dynamical degrees of rational maps.
Let $(\Gamma,\pi_1,\pi_2):X\ratt X$ be a rational self-correspondence
such that the restriction of $\pi_2$ to every irreducible component of
$\Gamma$ is dominant.
\begin{Def}
  In this case we say $\Gamma$ is a \newword{dominant} rational
    self-correspondence.
\end{Def}
Set $\Gamma^n:=\Gamma\circ\cdots\circ\Gamma$ ($n$
times). Each $[\Gamma^n]^*$ acts on $H^{2k}(X),$ preserving
$H^{k,k}(X)$. 
\begin{Def}
  Pick a norm $\norm{\cdot}$ on $H^{k,k}(X).$ The $k$th dynamical
  degree $\Theta_k$ of $\Gamma$ is
  $\lim_{n\to\infty}\norm{[\Gamma^n]^*}^{1/n}.$
\end{Def}
This limit exists (\cite{DinhSibony2005}), and is independent of the
choice of norm (\cite{Guedj2010}). In \cite{DinhSibony2005} and
\cite{DinhSibony2008}, Dinh and Sibony show that the topological
entropy of a rational map or rational correspondence is bounded from
above by the logarithm of its largest dynamical degree.

Suppose $H^{k,k}(X)=H^{2k}(X).$ Then, since pullback on $H^{2k}(X)$
is dual to pushforward on $H_{2k}(X),$ we can rewrite
\begin{align*}
  (\mbox{$k$th dynamical degree of
  $\Gamma$})=\lim_{n\to\infty}\norm{[\Gamma]_*:H_{2k}(X)\to H_{2k}(X)}^{1/n}.
\end{align*}

If $[\Gamma^n]^*=([\Gamma]^*)^n$ on $H^{k,k}(X)$ for all $n$, $\Gamma$
is called \newword{$k$-stable}. A rational correspondence that is
$k$-stable for all $k$ is called \newword{algebraically stable}.

For $k$-stable $\Gamma$, we can rewrite the dynamical degree
$\Theta_k$ as $\lim_{n\to\infty}\norm{([\Gamma]^*)^n}^{1/n},$ which is
the absolute value of the dominant eigenvalue of $[\Gamma]^*$. The
$k$th dynamical degree is hard to compute for rational maps or
correspondences that are not $k$-stable, except in a few examples.

\subsection{Example: The Cremona involution on $\P^2$}\label{sec:Cremona}
Let $g:\P^2\rat\P^2$ be the rational self-map given in coordinates by
\begin{align*}
  [\mathbf{x}:\mathbf{y}:\mathbf{z}]\mapsto[\mathbf{y}\mathbf{z}:\mathbf{x}\mathbf{z}:\mathbf{x}\mathbf{y}].
\end{align*}
The map $g$ is undefined at the coordinate points $[1:0:0],$
$[0:1:0],$ and $[0:0:1].$ Also, $g^2$ is the identity where defined
--- the complement of the coordinate lines. Let $\mathfrak{l}\in
H^{1,1}(\P^2)$ be the class of a line. Since $g$ is given by degree
two polynomials in the coordinates, it is easy to check that a general
line pulls back to a conic. So $g^*\mathfrak{l}=2\mathfrak{l}$ and
$g^*$ acts on $H^{1,1}(\P^2)$ via multiplication by 2. On the other
hand, $g^2=\Id_{\P^2}$, so $(g^2)^*$ acts by the identity on
$H^{1,1}(\P^2).$ In particular $(g^2)^*\ne(g^*)^2.$

Although $g$ is not $1$-stable, its dynamical degree $\Theta_1$ is
easily computable. For $n$ odd, $g^n=g$ on $\P^2$, and $(g^n)^*=g^*$
is multiplication by 2 on $H^{1,1}(\P^2).$ For $n$ even,
$g^n=\Id_{\P^2}$ and $(g^n)^*$ is the identity on $H^{1,1}(\P^2).$ For
any norm $\norm{\cdot}$, the sequence $\norm{(g^n)^*}^{1/n}$ goes to 1
as $n$ goes to $\infty,$ so $\Theta_1=1.$

We can understand the lack of $1$-stability of $g$ by
examining its graph $\Gr(g)\subseteq\P^2\times\P^2.$ Let the
coordinates on the first $\P^2$ factor be
$\mathbf{x}_1,\mathbf{y}_1,\mathbf{z}_1,$ and the coordinates on the
second $\P^2$ factor be $\mathbf{x}_2,\mathbf{y}_2,\mathbf{z}_2.$
Denote by $\pi_1$ and $\pi_2$ the projections onto the first and
second factors, respectively. Then $\Gr(g)$ is given by the equations
  \begin{align*}
    \mathbf{x}_1\mathbf{x}_2=\mathbf{y}_1\mathbf{y}_2=\mathbf{z}_1\mathbf{z}_2.
  \end{align*}
  Over the open set $U\subseteq\P^2$ where all coordinates are
  nonzero, $\Gr(g)$ has
  equations $$\mathbf{x}_1=\frac{1}{\mathbf{x}_2},\quad\mathbf{y}_1=\frac{1}{\mathbf{y}_2},\quad\mathbf{z}_1=\frac{1}{\mathbf{z}_2},$$
  so $\pi_1$ and $\pi_2$ are equal. The fibered product
  $V=\Gr(g)\thickspace{_{\pi_2}\times_{\pi_1}}\thickspace\Gr(g)$
  embeds naturally in $\P^2\times\P^2$ and has four irreducible
  components: 
  \begin{align*}
    V_{\mathrm{diag}}:&=\{[\mathbf{x}_1:\mathbf{y}_1:\mathbf{z}_1]=[\mathbf{x}_2:\mathbf{y}_2:\mathbf{z}_2]\}\\
    V_{\mathbf{x}}:&=\{\mathbf{x}_1=0\}\times\{\mathbf{x}_2=0\}\\
    V_{\mathbf{y}}:&=\{\mathbf{y}_1=0\}\times\{\mathbf{y}_2=0\}\\
    V_{\mathbf{z}}:&=\{\mathbf{z}_1=0\}\times\{\mathbf{z}_2=0\}.
  \end{align*}
  None of $V_\mathbf{x}$, $V_\mathbf{y},$ or $V_\mathbf{z}$ maps
  dominantly onto either $\P^2$ factor, so $V$ does not define a
  rational correspondence $\P^2\ratt\P^2.$ However, it induces the map
  $\mathfrak{l}\mapsto4\mathfrak{l}$ on $H^{1,1}(\P^2)$, which is the
  same as $(g^*)^2$. On the other hand, the graph of $g^2$ is
  $V_{\diag}$, just one of the four irreducible components of $V$.

  \subsection{Birationally conjugate rational correspondences}
  Let $(\Gamma,\pi_1,\pi_2):X\ratt X$ be a dominant rational
  self-correspondence, and let $\rho:X\rat X'$ be a birational
  equivalence. Then we obtain a dominant rational self-correspondence
  on $X'$ through conjugation by $\rho,$ as follows. Let $U$ be the
  domain of definition of $\rho$. Set
  $\Gamma'=\pi_1^{-1}(U)\cap\pi_2^{-1}(U).$ Then
  \begin{align*}
    (\Gamma',\rho\circ\pi_1,\rho\circ\pi_2):X'\ratt X'
  \end{align*}
  is a dominant rational self-correspondence.
  \begin{thm}[\cite{DinhSibony2005,DinhSibony2008,Truong2015}]\label{thm:BirationalInvariance}
    The dynamical degrees of $\Gamma$ and $\Gamma'$ are equal.
  \end{thm}
  \begin{rem}
    Theorem \ref{thm:BirationalInvariance} as stated does not appear in
    the references quoted above. The proofs in \cite{Truong2015} were
    modified in the Appendix to \cite{Ramadas2016} to give a complete
    proof.
  \end{rem}
  Thus we can study the dynamical degrees of $\Gamma:X\ratt X$ via the
  action of $\Gamma$ on the birational model $X'.$ The next two lemmas
  allow us to compare a rational correspondence on different
  birational models.
  \begin{lem}\label{lem:DynamicalDegreeInQuotient}
    Let $X$ and $X'$ be smooth projective varieties, with
    $H^{2k}(X)=H^{k,k}(X)$ and $H^{2k}(X')=H^{k,k}(X')$ for some
    $k$. Suppose $\Gamma:X\ratt X$ is a $k$-stable rational
    self-correspondence, $\Lambda\subseteq H_{2k}(X)$ is a subspace
    with $[\Gamma]_*(\Lambda)\subseteq\Lambda,$ and $X'$ admits a
    birational morphism $\rho:X\to X'$ such that
    $\Lambda\subseteq\ker(\rho_*)$. Set $\Omega=H_{2k}(X)/\Lambda.$
    Then the $k$th dynamical degree of $\Gamma$ is the absolute value
    of the dominant eigenvalue of the induced map
    $[\Gamma]_*:\Omega\to\Omega.$
  \end{lem}
  \begin{proof}
    Dynamical degrees are birational invariants (Theorem
    \ref{thm:BirationalInvariance}), so the $k$th dynamical degree of
    $\Gamma$ on $X$ is equal to its $k$th dynamical degree on
    $X'$. Denote by $[\Gamma^n]^{X}_*$ the pushforward induced by the
    $n$th iterate of $\Gamma$ on $H_{2k}(X),$ by
    $[\Gamma^n]^{\Omega}_*$ the induced map on $\Omega,$ and by
    $[\Gamma^n]^{X'}_*$ the pushforward on $H_{2k}(X').$ By the
    $k$-stability of $\Gamma$ on $X,$ we have
  \begin{align*}
    [\Gamma^n]^{X}_*&=([\Gamma]^{X}_*)^n\\
    [\Gamma^n]^{\Omega}_*&=([\Gamma]^{\Omega}_*)^n.
  \end{align*}
  Denote by $\pr$ the map from $H_{2k}(X)$ to $\Omega.$ Pick norms on
  $H_{2k}(X)$, $\Omega,$ and $H_{2k}(X').$ This induces norms on maps
  among these vector spaces. We abuse notation by using $\norm{\cdot}$
  to denote all of these norms. Since $\Lambda\subseteq\ker(\rho_*),$
  there is a factorization:
  \begin{center}
    \begin{tikzpicture}
      \matrix(m)[matrix of math nodes,row sep=2em,column
      sep=3em,minimum width=2em] {
        H_{2k}(X)&&H_{2k}(X')\\
        &\Omega&\\
      }; \path[-stealth] (m-1-1) edge node [above] {$\rho_*$}
      (m-1-3); \path[-stealth] (m-1-1) edge node [below left]
      {$\pr$} (m-2-2); \path[-stealth] (m-2-2) edge node
      [below right] {$\bar{\rho_*}$} (m-1-3);
    \end{tikzpicture}
  \end{center}
  For $n>0,$ we have
  \begin{align*}
    [\Gamma^n]^{X'}_*&=\rho_*\circ[\Gamma^n]^{X}_*\circ\rho^*\\
    &=\bar{\rho_*}\circ\pr\circ[\Gamma^n]^{X}_*\circ\rho^*\\
    &=\bar{\rho_*}\circ[\Gamma^n]^{\Omega}_*\circ\pr\circ\rho^*\\
    &=\bar{\rho_*}\circ([\Gamma]^{\Omega}_*)^n\circ\pr\circ\rho^*.
  \end{align*}
  Thus by submultiplicativity of the induced norms:
  \begin{align*}
    \norm{[\Gamma^n]^{X'}_*}&\le\norm{\bar{\rho_*}}\norm{\pr}\norm{\rho^*}\norm{([\Gamma]^{\Omega}_*)^n}.
  \end{align*}
  Taking $n$th roots and the limit as $n\to\infty$, we obtain:
  \begin{align*}
    \mbox{($k$th dynamical degree of $\Gamma$)}\le\abs{\mbox{dominant
        eigenvalue of $[\Gamma]^{\Omega}_*$}}.
  \end{align*}
On the other hand, since $\Gamma$ is $k$-stable on $X$,
\begin{align*}
  \mbox{($k$th dynamical degree of $\Gamma$)}&=\abs{\mbox{dominant
      eigenvalue of $[\Gamma]^{X}_*$}}\\
  &\ge\abs{\mbox{dominant
      eigenvalue of $[\Gamma]^{\Omega}_*$}}.\qedhere
\end{align*}
\end{proof}
\begin{lem}\label{lem:CriterionForComposability}
  Let $\Gamma':X_2\ratt X_3$ be a rational correspondence. For
  $j\in\{2,3\},$ let $X'_j$ be a smooth projective variety admitting a
  birational morphism $\rho_j$ from $X_j$. Suppose for fixed $k$,
  $[\Gamma']_*:H_{2k}(X_2)\to H_{2k}(X_3)$ takes $\ker((\rho_2)_*)$
  to $\ker((\rho_3)_*).$ Then
  \begin{enumerate}[(i)]
  \item The following diagram commutes:
    \begin{center}
      \begin{tikzpicture}
        \matrix(m)[matrix of math nodes,row sep=2em,column
        sep=3em,minimum width=2em] {
          H_{2k}(X_2)&H_{2k}(X_3)\\
          H_{2k}(X'_2)&H_{2k}(X'_3)\\
        }; \path[-stealth] (m-1-1) edge node [above] {$[\Gamma']_*$}
        (m-1-2); \path[-stealth] (m-1-1) edge node [left]
        {$(\rho_2)_*$} (m-2-1); \path[-stealth] (m-2-1) edge node
        [above] {$[\Gamma']_*$} (m-2-2); \path[-stealth] (m-1-2) edge
        node [right] {$(\rho_3)_*$} (m-2-2);
      \end{tikzpicture}
    \end{center}
    Thus $[\Gamma']_*:H_{2k}(X'_2)\to H_{2k}(X'_3)$ can be identified
    with the induced map
    $$[\Gamma']_*:H_{2k}(X_2)/\ker((\rho_2)_*)\to H_{2k}(X_3)/\ker((\rho_3)_*),$$\label{item1}
  \item If $[\Gamma']_*:H_{2k}(X_2)\to H_{2k}(X_3)$ takes effective
    classes to effective classes, then so does
    $[\Gamma']_*:H_{2k}(X'_2)\to H_{2k}(X'_3)$, and\label{item2}
  \item If $\Gamma:X_1\ratt X_2$ is $2k$-homologically composable
    with $\Gamma':X_2\ratt X_3$, then $\Gamma:X_1\ratt X'_2$ is
    $2k$-homologically composable with $\Gamma':X'_2\ratt
    X'_3$.\label{item3}
  \end{enumerate}
\end{lem}
\begin{proof}

  Since we discuss the same rational correspondence on different
  spaces, we include a superscript in the notation for pushforward;
  e.g. we denote by $[\Gamma']_*^{X_2,X_3}$ the induced map from
  $H_{2k}(X_2)$ to $H_{2k}(X_3)$, and we denote by
  $[\Gamma']_*^{X'_2,X'_3}$ the induced map from $H_{2k}(X'_2)$ to
  $H_{2k}(X'_3)$.
  
  Since the morphism $\rho_2$ is birational, $(\rho_2)_*\circ\rho_2^*$
  is the identity on $H_{2k}(X'_2).$ Thus:
  \begin{align*}
    \im(\rho_2^*\circ(\rho_2)_*-\Id)\subseteq\ker((\rho_2)_*).
  \end{align*}
  By assumption $[\Gamma']^{X_2,X_3}_*$ sends $\ker((\rho_2)_*)$ to
  $\ker((\rho_3)_*).$ Thus
  \begin{align}
    [\Gamma']^{X'_2,X'_3}_*\circ(\rho_2)_*&=(\rho_3)_*\circ[\Gamma']^{X_2,X_3}_*\circ\rho_2^*\circ(\rho_2)_*\nonumber\\
                                            &=(\rho_3)_*\circ[\Gamma']^{X_2,X_3}_*\label{eqn:Two}.
  \end{align}
  This proves \eqref{item1}.

  For $\alpha$ an effective $k$-dimensional cycle on $X'_2$, there is
  an effective cycle $\tilde{\alpha}$ on $X_2$ satisfying
  $(\rho_2)_*(\tilde{\alpha})=\alpha$. Thus
  $\tilde{\alpha}-\rho_2^*(\alpha)\in\ker((\rho_2)_*).$ Again, since
  $[\Gamma']^{X_2,X_3}_*$ sends $\ker((\rho_2)_*)$ to
  $\ker((\rho_3)_*),$
  \begin{align*}
    [\Gamma']^{X'_2,X'_3}_*(\alpha)&=(\rho_3)_*\circ[\Gamma']^{X_2,X_3}_*\circ(\rho_2^*)(\alpha)\\
&=(\rho_3)_*\circ[\Gamma']^{X_2,X_3}_*(\tilde{\alpha}).
  \end{align*}
  By assumption, $[\Gamma']^{X_2,X_3}_*$ preserves
  effectiveness. The pushforward by the regular map $\rho_3$
  preserves effectiveness, so $[\Gamma']^{X'_2,X'_3}_*(\alpha)$ is
  effective, proving \eqref{item2}. Finally, for \eqref{item3},
  \begin{align*}
    [\Gamma']^{X'_2,X'_3}_*\circ[\Gamma]^{X_1,X'_2}_*&=(\rho_3)_*\circ[\Gamma']^{X_2,X_3}_*\circ\rho_2^*\circ(\rho_2)_*\circ[\Gamma]^{X_1,X_2}_*.\\
    &=(\rho_3)_*\circ[\Gamma']^{X_2,X_3}_*\circ[\Gamma]^{X_1,X_2}_*\quad\quad\quad\quad\mbox{by
      \eqref{eqn:Two}}\\
    &=(\rho_3)_*\circ[\Gamma'\circ\Gamma]^{X_1,X_3}_*\\
    &=[\Gamma'\circ\Gamma]^{X_1,X'_3}_*.\qedhere
  \end{align*}
\end{proof}

\section{$\M_{0,N}$ and Hurwitz correspondences}
The moduli space $\M_{0,N}$ parametrizes all ways of labeling $N$
distinct points on $\P^1$, up to projective change of coordinates.
\begin{Def}
  An \newword{$N$-marked smooth genus zero curve} is a curve $C$,
  isomorphic to $\P^1$, together with distinct labeled points
  $p_1,\ldots,p_N\in C$. For $\bP$ a finite set, an
  \newword{$\bP$-marked smooth genus zero curve} is a curve $C$,
  isomorphic to $\P^1$, together with an injective map $\bP\into C.$
\end{Def}

\begin{Def}
  Let $N\ge3.$ There is a smooth quasiprojective variety $\M_{0,N}$ of
  dimension $N-3$ parametrizing all $N$-marked smooth genus zero
  curves up to isomorphism. Likewise, for $\bP$ a finite set with
  $\abs{\bP}\ge3,$ there is a moduli space
  $\M_{0,\bP}\cong\M_{0,\abs{\bP}}$, parametrizing smooth genus zero
  $\bP$-marked curves up to isomorphism.
\end{Def}

\begin{Def}
  For $N\ge N'\ge3,$ there is a \newword{forgetful map}
  $\mu:\M_{0,N}\to\M_{0,N'}$ sending $(C,p_1,\ldots,p_N)$ to
  $(C,p_1,\ldots,p_{N'}).$ Similarly, for $\bP'\into\bP$ with
  $\abs{\bP'}\ge3,$ there is a forgetful map
  $\mu:\M_{0,\bP}\to\M_{0,\bP'}$.
\end{Def}

Hurwitz spaces are moduli spaces parametrizing finite maps with
prescribed ramification between smooth curves of prescribed genus.  In
this paper, we only deal with Hurwitz spaces of maps between genus
zero curves. See \cite{RomagnyWewers2006} for a summary and proofs of
facts quoted here.

\begin{Def}[\emph{Hurwitz space}, \cite{Ramadas2016}, Definition 2.2]\label{Def:HurwitzSpace}
  Fix discrete data:
  \begin{itemize}
  \item $\bA$ and $\bB$ finite sets with cardinality at least 3
    (marked points on source and target curves, respectively),
  \item $d$ a positive integer (degree),
  \item $F:\bA\to \bB$ a map,
  \item $\br:\bB\to\{\mbox{partitions of $d$}\}$ (branching), and
  \item $\rm:\bA\to\Z^{>0}$ (ramification),
  \end{itemize}
  such that
  \begin{itemize}
  \item (Condition 1, Riemann-Hurwitz constraint) $\sum_{b\in\bB}\left(d-\mbox{length of
        $\br(b)$}\right)=2d-2$, and
  \item (Condition 2) for all $b\in\bB,$ the multiset $(\rm(a))_{a\in
      F^{-1}(b)}$ is a submultiset of $\br(b)$.
  \end{itemize}  
  There exists a smooth quasiprojective variety
  $\cH=\cH(\bA,\bB,d,F,\br,\rm),$ a \newword{Hurwitz space},
  parametrizing morphisms $f:C\to D$ up to isomorphism, where
  \begin{itemize}
  \item $C$ and $D$ are $\bA$-marked and $\bB$-marked smooth connected
    genus zero curves, respectively,
  \item $f$ is degree $d$,
  \item for all $a\in\bA,$ $f(a)=F(a)$ (via the injections $\bA\into
    C$ and $\bB\into D$),
    \item for all $b\in\bB,$ the branching of $f$ over $b$ is given by
    the partition $\br(b)$, and
  \item for all $a\in\bA,$ the local degree of $f$ at $a$ is equal to
    $\rm(a)$.  
  \end{itemize}
\end{Def}
\begin{figure}
  \centering
\begin{align*}
    \begin{array}{c}
      \H=\left\{\raisebox{-50pt}{
      \begin{tikzpicture}[scale=0.4]
        \draw (0,0) node
        {$\bullet$}; \draw (0,0) node [below]
        {\raisebox{-12pt}{$b_1$}}; \node at (1.7,0)
        {$\bullet$};\draw (1.7,0) node [below]
        {\raisebox{-12pt}{$b_2$}}; \node at (4.3,0)
        {$\bullet$};\draw (4.3,0) node [below]
        {\raisebox{-12pt}{$b_3$}}; \node at (6,0)
        {$\bullet$};\draw (6,0) node [below]
        {\raisebox{-12pt}{$b_4$}}; \node at (0,4.2)
        {$\bullet$};\draw (0,4.2) node [left]
        {$a_1\thickspace$}; \node at (1.7,4.3)
        {$\bullet$};\draw (1.7,4.3) node [right]
        {$\thickspace
          a_2$}; \node at (4.3,5.7)
        {$\bullet$};\draw (4.2,5.6) node [right] {$\thinspace
          a_3$}; \node at (4.3,3.4)
        {$\bullet$};\draw (4.3,3.45) node [below]
        {\raisebox{-5pt}{$a_4$}}; \draw[very thick] (-4,0) -- (10,0);
        \draw[->] (3,2.8) -- (3,1) node [midway,left]
        {$f\thinspace$}; \node at (11,5)
        {$C$}; \node at (11,0)
        {$D$}; \draw[very thick] plot [smooth, tension=1.2]
        coordinates { (-4,3.2) (0,4.2) (-4,5.2) }; \draw[very thick]
        plot [smooth, tension=1] coordinates {(-4,7) (4,6) (2,4)
          (10,3)}; \draw[very thick] plot [smooth, tension=1.2]
        coordinates{ (10,6.8) (6,5.8) (10,4.8)}; \draw[very thick]
        (12,-2)--(12,6.5);
      \end{tikzpicture}
      }
      \begin{array}{l}
        \mbox{$f:C\to D$ degree 3}\\
        a_1\xrightarrow{2}{}b_1\\
        a_2\xrightarrow{2}{}b_2\\
        a_3\xrightarrow{2}{}b_3\\
        a_4\xrightarrow{1}{}b_3\\
        \mbox{$(1,2)$-branching over $b_4$}
      \end{array}
      \right\}
    \end{array}
  \end{align*}
  \captionsetup{singlelinecheck=off}
   \caption[H]{The Hurwitz space $\H=\H(\bA,\bB,d,F,\br,rm)$, where
    \begin{itemize}
    \item $\bA=\{a_1,a_2,a_3,a_4\},$
    \item $\bB=\{b_1,b_2,b_3,b_4\},$
    \item $d=3$,
    \item $F(a_1)=b_1,$ $F(a_2)=b_2,$ $F(a_3)=F(a_4)=b_3$,
    \item $\br(b_i)=(1,2)$ for $i=1,\ldots,4$, and
    \item $\rm(a_1)=\rm(a_2)=\rm(a_3)=2,$ $\rm(a_4)=1.$
    \end{itemize}
    \textcolor{white}{SpaceText}
  }
  \label{fig:HurwitzSpace}
\end{figure}
\begin{rem}
  One may construct $\H$ as follows
  (\cite{HarrisMumford1982}). Consider $\M_{0,\bA}\times\M_{0,\bB}$
  with its two universal curves $\mathcal{U}_{0,\bA}$ and
  $\mathcal{U}_{0,\bB}.$ Let $\Hilb$ be the relative Hilbert scheme of
  degree $d$ morphisms $\cU_{0,\bA}\to\cU_{0,\bB}.$ The conditions
  that $a\in\bA$ map to $F(a)$ with local degree $\rm(a)$ and that the
  branching over $b\in\bB$ be given by $\br(b)$ are locally
  closed. Thus $\H$ is a locally closed subvariety of $\Hilb.$
\end{rem}

The Hurwitz space $\cH$ admits a map $\pi_{\bA}$ to $\M_{0,\bA},$
sending $[f:C\to D]$ to the marked source curve $[C]$. It similarly
admits a map $\pi_{\bB}$ to $\M_{0,\bB}$, sending $[f:C\to D]$ to the
marked target curve $[D]$. The space $\cH$ may be empty; if not, the
``target curve'' map $\pi_{\bB}$ is a finite covering map and
$\pi_{\bA}\circ\pi_{\bB}^{-1}$ defines a multivalued map from
$\M_{0,\bA}$ to $\M_{0,\bB}$. If $X_{\bB}$ and $X_{\bA}$ are
projective compactifications of $\M_{0,\bB}$ and $\M_{0,\bA}$
respectively, $(\H,\pi_{\bB},\pi_{\bA}):X_{\bB}\ratt X_{\bA}$ is a
rational correspondence. We generalize this as follows.
\begin{Def}[\emph{Hurwitz correspondence}, \cite{Ramadas2016},
  Definition 2.19]
  Let $\bA'$ be any subset of $\bA$ with cardinality at least 3. There
  is a forgetful morphism $\mu:\M_{0,\bA}\to\M_{0,\bA'}$. Let $\Gamma$
  be a union of connected components of $\cH$. If $X_{\bA'}$ and
  $X_{\bB}$ are smooth projective compactifications of $\M_{0,\bA'}$
  and $\M_{0,\bB}$ respectively,
  then
  $$\left(\Gamma,\pi_{\bB},\mu\circ\pi_{\bA}\right):X_{\bB}\ratt
  X_{\bA'}$$
  is a rational correspondence. We call such a rational correspondence
  a \newword{Hurwitz correspondence}.
\end{Def}

\subsection{Connection to dynamics on $\P^1$ and the Thurston pullback
map}\label{sec:Teichmuller}
This section is purely for motivation. We summarize parts of the
results in \cite{Koch2013}. Let $S^2$ denote an oriented 2-sphere.
\begin{Def}
  Let $\mathbf{P}\subset S^2$ be a finite set. Define an equivalence
  relation on orientation-preserving homeomorphisms $S^2\to\P^1$ as
  follows: $\psi_1$ is equivalent to $\psi_2$ if there exists
  $\xi\in\Aut(\P^1)$ such that $\psi_1$ and $\xi\circ\psi_2$ agree on
  $\mathbf{P}$ and are isotopic relative to $\mathbf{P}$. The
  \newword{Teichm\"uller space} $\mathcal{T}(S^2,\mathbf{P})$ of
  $(S^2,\mathbf{P})$ is the set of equivalence classes of
  homeomorphisms.
\end{Def}
Teichm\"uller space has a natural structure as a noncompact
nonalgebraic complex manifold; it is isomorphic to a bounded domain in
$\C^{\abs{\mathbf{P}}-3}.$ Given an element
$[\psi]\in\mathcal{T}(S^2,\mathbf{P})$, the restriction
$\psi|_{\mathbf{P}}:\mathbf{P}\to\P^1$ defines an element of
$\M_{0,\mathbf{P}}.$ This gives rise to a map of complex manifolds
\begin{align*}
  \mathcal{T}(S^2,\mathbf{P})\xrightarrow{\cv}\M_{0,\mathbf{P}}.
\end{align*}
In fact, this is a covering map, realizing
$\mathcal{T}(S^2,\mathbf{P})$ as the universal cover of $\M_{0,\bP}.$
\begin{Def}
  An orientation-preserving branched covering $\phi$ from $S^2$ to
  itself is called \newword{postcritically finite} if the
  post-critical set $$\mathbf{P}:=\{\phi^n(x)|n>0,\mbox{$x$ a critical
    point of $\phi$}\}$$ is finite.
\end{Def}

We say $\phi$ is \newword{combinatorially equivalent} to an algebraic
morphism if there exist orientation-preserving homeomorphisms
$\psi_1,\psi_2:S^2\to\P^1$ such that
\begin{align*}
  \psi_2\circ\phi\circ\psi_1^{-1}:\P^1\to\P^1
\end{align*}
is an algebraic morphism, and $\psi_1$ and $\psi_2$ are isotopic
relative to the postcritical set $\mathbf{P}.$ W. Thurston gave a
topological characterization for $\phi$ to be combinatorially
equivalent to an algebraic morphism, in terms of curve systems on
$S^2\setminus\mathbf{P}.$ This characterization can also be stated in
terms of a self-map on $\mathcal{T}(S^2,\mathbf{P})$
(\cite{DouadyHubbard1993}). If
$\psi:(S^2,\mathbf{P})\to(\P^1,\psi(\mathbf{P}))$ is an
orientation-preserving homeomorphism then the induced complex
structure on $S^2\setminus\mathbf{P}$ can be pulled back via the
covering map $\phi|_{S^2\setminus\phi^{-1}(\mathbf{P})}$ to obtain
another complex structure on $S^2\setminus\phi^{-1}(\mathbf{P}).$ This
extends to a complex structure on all of $S^2,$ and yields another
homeomorphism $\psi':(S^2,\mathbf{P})\to(\P^1,\psi'(\mathbf{P}))$ such
that
\begin{align*}
  \psi\circ\phi\circ(\psi')^{-1}:(\P^1,\psi'(\mathbf{P}))\to(\P^1,\psi(\bP))
\end{align*}
is an algebraic morphism. The map
\begin{align*}
  \mathcal{T}(S^2,\mathbf{P})&\xrightarrow{\Thurst(\phi)}\mathcal{T}(S^2,\mathbf{P})\\
  [\psi]&\mapsto[\psi']
\end{align*}
is a well-defined holomorphic map (\cite{DouadyHubbard1993}) called the
\newword{Thurston pullback map}.

\medskip

This pullback map is a contraction in the hyperbolic metric on
$\mathcal{T}(S^2,\mathbf{P});$ it is a theorem of Thurston that $\phi$
is combinatorially equivalent to an algebraic map if and only if
$\Thurst(\phi)$ has a fixed point. The dynamics of $\Thurst(\phi)$ are
thus of interest.

The Thurston pullback map is holomorphic, but not algebraic. However,
Koch (\cite{Koch2013}) showed that it descends to a Hurwitz
correspondence on the algebraic variety $\M_{0,\mathbf{P}}$.

Given $\phi:(S^2,\mathbf{P})\to(S^2,\mathbf{P}),$ denote by $d$ the
topological degree of $\phi$. Define
\begin{align*}
  \br:\mathbf{P}&\to\{\mbox{partitions of $d$}\}\\
  p&\mapsto\mbox{branching of $\phi$ over $p$}
\end{align*}
and
\begin{align*}
  \rm:\mathbf{P}&\to\Z^{>0}\\
  p&\mapsto\mbox{local degree of $\phi$ at $p$}.
\end{align*}
The data
$(\mathbf{P},\mathbf{P},d,\phi|_{\mathbf{P}}:\mathbf{P}\to\mathbf{P},\br,\rm)$
satisfy Conditions 1 and 2 in Definition
\ref{Def:HurwitzSpace}. Denote by $\H$ the Hurwitz space
$\H(\mathbf{P},\mathbf{P},d,\phi|_{\mathbf{P}}:\mathbf{P}\to\mathbf{P},\br,\rm)$.
Given a homeomorphism
$\psi:(S^2,\mathbf{P})\to(\P^1,\psi(\mathbf{P})),$ there exists a
homeomorphism $\psi':(S^2,\bP)\to(\P^1,\psi'(\bP)),$ with $[\psi']=\Thurst(\phi)([\psi])$,
and such that
\begin{align*}
  \psi\circ\phi\circ(\psi')^{-1}:(\P^1,\psi'(\mathbf{P}))\to(\P^1,\psi(\mathbf{P}))
\end{align*}
is an algebraic morphism. This defines a point in $\H$. By
\cite{Koch2013}, we obtain a holomorphic covering map
\begin{align*}
  \mathcal{T}(S^2,\mathbf{P})&\to\H\\
  [\psi]&\mapsto[(\P^1,\psi'(\mathbf{P}))\xrightarrow{\psi\circ\phi\circ(\psi')^{-1}}(\P^1,\psi(\mathbf{P}))]
\end{align*}
whose image is a connected component $\Gamma$ of $\H$. We have
the commutative diagram:
\begin{center}
  \begin{tikzpicture}
    \matrix(m)[matrix of math nodes,row sep=3em,column sep=8em,minimum
    width=2em] {
      \mathcal{T}(S^2,\mathbf{P})&&\mathcal{T}(S^2,\mathbf{P})\\
      &\Gamma&\\
      \M_{0,\mathbf{P}}&&\M_{0,\mathbf{P}}\\
    };
    \path[-stealth] (m-1-1) edge node [above] {$\Thurst(\phi)$} (m-1-3);
    \path[-stealth] (m-1-1) edge (m-2-2);
    \path[-stealth] (m-1-1) edge node [left] {$\cv$} (m-3-1);
    \path[-stealth] (m-1-3) edge node [left] {$\cv$} (m-3-3);
    \path[-stealth] (m-2-2) edge node [above left]
    {$\pi_\mathbf{P}^{\mathrm{target}}$} (m-3-1);
    \path[-stealth] (m-2-2) edge node [above right] {$\pi_\mathbf{P}^{\mathrm{source}}$} (m-3-3);
  \end{tikzpicture}
\end{center}
Thus the Hurwitz correspondence $\Gamma$ is an algebraic ``shadow'' of
the nonalgebraic Thurston pullback map. Our study of the dynamics of
Hurwitz self-correspondences is motivated in part by this connection
to the dynamics of the Thurston pullback map.

\section{Compactifications of $\M_{0,N}$}
$\M_{0,N}$ parametrizes smooth curves of genus zero with $N$ distinct
marked points.  However, when $N>3$ there are 1-parameter families
$C(\mathbf{t})_{\mathbf{t}\ne0}$ of smooth curves with $N$ distinct
marked points such that as $\mathbf{t}$ goes to zero, there is no
limiting smooth curve where the marked points remain distinct. Thus
$\M_{0,N}$ cannot be compact. There are many compactifications of
$\M_{0,N}$ that are essentially based on describing what happens when
marked points collide.

$\Mbar_{0,N}$ is the most widely studied of these
compactifications. Here, the marked points are always distinct, but
the curve may be nodal. Moduli spaces of weighted stable curves are a
generalization of $\Mbar_{0,N}$ constructed by Hassett in
\cite{Hassett2003}. Here, marked points are assigned weights between 0
and 1, and a set of marked points may coincide as long as their total
weight is not more than 1.

\subsection{ $\Mbar_{0,N}$ and its combinatorial
  structure}\label{sec:StableCurves}
We refer the reader to \cite{KockVainsencher} for an extended
discussion.
\begin{Def}\label{Def:StableCurve}
  A \newword{stable $N$-marked genus zero curve} is a connected
  algebraic curve $C$ of arithmetic genus zero whose only
  singularities are simple nodes, together with $N$ distinct smooth
  marked points $p_1,\ldots,p_N$ on $C$, such that the set of
  automorphisms $C\to C$ that fix every marked point $p_i$ is finite.
\end{Def}
The irreducible components of such $C$ are all isomorphic to
$\P^1$. Points of $C$ that are either marked points or nodes are
called \newword{special points}. The \newword{stability} condition
that $(C,p_1,\ldots,p_N)$ has finitely many automorphisms implies that
it has no nontrivial automorphisms, and is equivalent to the condition
that every irreducible component of $C$ has at least three special
points.
\begin{thm}[Deligne-Mumford, Grothendieck, Knudsen]
  Let $N\ge3.$ There is a smooth projective variety $\Mbar_{0,N}$ of
  dimension $N-3$ that is a fine moduli space for stable $N$-marked
  genus zero curves. It contains $\M_{0,N}$ as a dense open subset.
\end{thm}
For $\bP$ a finite set, we analogously define stable genus zero
$\bP$-marked curves, and their moduli space $\Mbar_{0,\bP}$
(isomorphic to $\Mbar_{0,\abs{\bP}}$). It contains $\M_{0,\bP}$ as a
dense open subset.

The \newword{boundary} $\Mbar_{0,N}\setminus\M_{0,N}$ is a simple
normal crossings divisor. Points on the boundary correspond to
reducible stable curves. The topological type of a stable curve is
captured by the combinatorial information of its dual tree. This
classification of stable curves by topological type gives a
stratification of $\Mbar_{0,N}$.
\begin{Def}\label{Def:DualTree}
  Let $(C,p_1,\ldots,p_N)$ be a stable genus zero curve. Its
  \newword{dual tree} $\sigma$ is the graph defined as follows. The
  vertices $v$ of $\sigma$ correspond to the irreducible components
  $C_v$ of $C$. Two vertices $v_1$ and $v_2$ are connected by an edge
  if and only if the components $C_{v_1}$ and $C_{v_2}$ meet at a
  node. For each marked point $p_i$ on $C_v,$ we attach a
  \newword{leg} $\ell_i$ to the vertex $v.$ Note that edges and legs
  are distinct from each other. The graph $\sigma$ is a tree because
  $C$ has arithmetic genus zero.
\end{Def}
For a vertex $v$ on $\sigma,$ set
$$\Delta_v=\{\mbox{Legs attached to
  $v$}\}\cup\{\mbox{edges incident to $v$}\}.$$
We refer to elements of $\Delta_v$ as \newword{flags} on $v$. We
define the \newword{valence} of $v$, denoted $\abs{v}$, to be the
cardinality of $\Delta_v.$ For $i\in\{1,\ldots,N\}$, we define
$\delta(v\to i)$ to be the unique flag in $\Delta_v$ that connects the
leg $\ell_i$ to $v$, i.e. is part of the unique non-repeating path in
$\sigma$ from $v$ to $\ell_i.$ If $\ell_i\in\Delta_v,$ then
$\delta(v\to i)=\ell_i;$ otherwise $\delta(v\to i)$ is an edge.  There
is a canonical injection $\Delta_v\into C_v$ whose image is the set of
special points of $C_v$. Thus $C_v$ is a $\Delta_v$-marked smooth
genus zero curve. We denote by $\mathcal{V}(\sigma)$ the set of
vertices of $\sigma$. We define the \newword{moduli dimension}
$\md(v)$ of $v\in\mathcal{V}(\sigma)$ to be $\abs{v}-3.$

\begin{Def}
  A \newword{stable $N$-marked tree} is a tree $\sigma$ with marked
  legs $\ell_1,\ldots,\ell_N$ such that every vertex has valence at
  least 3.
\end{Def}
For fixed $N$, there are finitely many isomorphism classes of stable
$N$-marked trees, and each of these arises as the dual tree of some
stable $N$-marked genus zero curve.
\begin{Def}\label{Def:BoundaryStratum}
  Given $\sigma$ a stable $N$-marked tree, the closure $S_\sigma$ of
  the set $\{[(C,p_1,\ldots,p_N)]:\mbox{$C$ has dual graph
    $\sigma$}\}$ is an irreducible subvariety of $\Mbar_{0,N}$. These
  special subvarieties are called \newword{boundary strata}.
\end{Def}
Boundary strata on $\Mbar_{0,N}$ are in bijection with isomorphism
classes of stable $N$-marked trees.

Boundary strata on $\Mbar_{0,N}$ decompose into products of
smaller-dimensional spaces of stable curves. Let $S_\sigma$ be a
boundary stratum in $\Mbar_{0,N}$. Given stable curves
$([C_v]\in\Mbar_{0,\Delta_v})_{v\in\cV(\sigma)}$, we can glue the
curves $C_v$ together to obtain a curve $C$ in $S_\sigma$ as
follows. Whenever there is an edge between $v_1$ and $v_2$, glue
$C_{v_1}$ to $C_{v_2}$ at the corresponding marked point of each
curve. This procedure defines a \newword{gluing morphism}
\begin{align*}
  \prod_{v\in\mathcal{V}(\sigma)}\Mbar_{0,\Delta_v}\cong S_\sigma\into\Mbar_{0,N}.
\end{align*}
So
\begin{align*}
  \dim S_\sigma=\sum_{v\in\mathcal{V}(\sigma)}(\abs{v}-3)=\sum_{v\in\mathcal{V}(\sigma)}\md(v).
\end{align*}
The codimension of $S_\sigma$ in $\Mbar_{0,N}$ equals the number of edges of
$\sigma.$

\medskip

\noindent\textbf{Forgetful maps.} For $N\ge N'\ge3,$
the forgetful map $\mu:\M_{0,N}\to\M_{0,N'}$ extends to
$\mu:\Mbar_{0,N}\to\Mbar_{0,N'}$. Let $(C,p_1,\ldots,p_N)$ be any
$N$-marked stable curve. The curve $(C,p_1,\ldots,p_{N'})$ obtained by
forgetting $p_{N'+1},\ldots,p_N$ may not be stable. However, we
obtain $(C',p_1,\ldots,p_{N'})$ a stable $N'$-marked curve by
\emph{stabilizing}, i.e. successively contracting components of $C$
with fewer than 3 special points. The map $\mu$ sends
$[(C,p_1,\ldots,p_N)]$ to $[(C',p_1,\ldots,p_{N'})].$

If $\sigma$ is the dual tree of $C$, we obtain the dual tree $\sigma'$
of $C'$ by deleting legs $\ell_{N'+1},\ldots,\ell_N$ and applying a
finite sequence of steps, also called stabilization. Each step is
either of the form:
\begin{itemize}
\item For a vertex $v$ of valence 1, delete $v$, together with its
  incident edge, or
\item For a vertex $v$ of valence 2 with edges $e_1$ and $e_2$
  connecting $v$ to vertices $v_1$ and $v_2$, respectively, delete $v$
  together with $e_1$ and $e_2,$ and connect $v_1$ to $v_2$ by an
  edge, or
\item For a vertex $v$ of valence 2 incident to an edge $e$ connecting
  $v$ to another vertex $v_1$, and also incident to a leg $\ell_i$,
  delete $v$ together with $e$, and attach $\ell_i$ to $v_1$.
\end{itemize}
The deletion of the vertex $v$ on $\sigma$ corresponds to the
contraction of $C_v$ on $C$.

\medskip

\noindent\textbf{Homology groups of $\Mbar_{0,N}$.} The homology group
$H_{2k}(\Mbar_{0,N})$ is isomorphic to the Chow group
$A_k(\Mbar_{0,N})$ and is generated by the classes of $k$-dimensional
boundary strata. $\Mbar_{0,N}$ has no odd-dimensional homology (Keel,
\cite{Keel1992}).

Additive relations among boundary strata are described in
\cite{KontsevichManin1994} by Kontsevich and Manin. Let $\sigma$ be
the dual tree of some $(k+1)$-dimensional boundary stratum, let $v$ be
a vertex on $\sigma$ with valence at least four, and let
$i_1,i_2,i_3,i_4\in\{1,\ldots,N\}$ be such that the flags $\delta(v\to
i_1),\ldots,\delta(v\to i_4)$ on $v$ are distinct. The data
$\sigma,v,i_1,\ldots,i_4$ determine a relation
$R(\sigma,v,i_1,\ldots,i_4)$ among $k$-dimensional boundary
strata. For every set partition $\Delta_v\setminus\{\delta(v\to
i_1),\ldots,\delta(v\to i_4)\}=\Delta_1\sqcup\Delta_2$, define a
stable tree $\sigma(i_1i_2\Delta_1|\Delta_2i_3i_4)$ as follows.
\begin{enumerate}
\item Insert an edge into $v$, splitting it into two vertices $v_1$
  and $v_2$.
\item Attach the flags in $\Delta_1\cup\{\delta(v\to i_1),\delta(v\to i_2)\}$ to $v_1$.
\item Attach the flags in $\Delta_2\cup\{\delta(v\to i_3),\delta(v\to i_4)\}$ to $v_2$.
\end{enumerate}
The tree $\sigma(i_1i_2\Delta_1|\Delta_2i_3i_4)$ has
one more edge than $\sigma$ and thus corresponds to a $k$-dimensional
boundary stratum $S(i_1i_2\Delta_1|\Delta_2i_3i_4)$. The tree
$\sigma(i_1i_3\Delta_1|\Delta_2i_2i_4)$ and stratum
$S(i_1i_3\Delta_1|\Delta_2i_2i_4)$ are defined analogously. Then
\begin{align}\label{eq:KontsevichManin}
  R(\sigma,v,i_1,\ldots,i_4):&=\sum_{(\Delta_1,\Delta_2)}[S(i_1i_2\Delta_1|\Delta_2i_3i_4)]-\sum_{(\Delta_1,\Delta_2)}[S(i_1i_3\Delta_1|\Delta_2i_2i_4)]\\
  &=0\in H_{2k}(\Mbar_{0,N})\mbox{\quad(resp. $A_k(\Mbar_{0,N})$)}\nonumber,
\end{align}
where the sum is over set partitions $\Delta_v\setminus\{\delta(v\to
i_1),\ldots,\delta(v\to i_4)\}=\Delta_1\cup\Delta_2$. These relations,
varying $\sigma,v,i_1,\ldots,i_4$, generate all additive relations among
$k$-dimensional boundary strata.

\subsection{Weighted stable curves}\label{sec:WeightedStableCurves}
\begin{Def}[Hassett, \cite{Hassett2003}]
  A \newword{weight datum} is a tuple
  $\be=(\epsilon_1,\ldots,\epsilon_N)\in(\Q\cap(0,1])^N$
  such that $\sum_{i=1}^N\epsilon_i>2.$
\end{Def}
\begin{Def}[Hassett, \cite{Hassett2003}]
  Let $\be$ be a weight datum. A \newword{genus zero
    $\be$-stable curve} is a connected algebraic curve $C$ of
  arithmetic genus zero, whose only singularities are simple nodes,
  together with smooth marked points $p_1,\ldots,p_N$, not necessarily
  distinct, such that
  \begin{enumerate}[(1)]
  \item If $p_{i_1}=\cdots=p_{i_s}$ then
    $\epsilon_{i_1}+\cdots+\epsilon_{i_s}\le1$, and
  \item For any irreducible component $C_v$,
    \begin{align*}
      \#\{\mbox{nodes on $C_v$}\}+\sum_{\text{$p_i$ on
          $C_v$}}\epsilon_i>2.
    \end{align*}
  \end{enumerate}
\end{Def}
\begin{Defthm}[Hassett, \cite{Hassett2003}]
  Given a weight datum $\be,$ there is a smooth projective
  variety $\Mbar_{0,N}(\be)$ of dimension $N-3$ that is a
  fine moduli space for $\be$-stable genus zero curves. It
  contains $\M_{0,N}$ as a dense open set. There is a
  \newword{reduction morphism}
  $\rho_{\be}:\Mbar_{0,N}\to\Mbar_{0,N}(\be)$ that
  respects the open inclusion of $\M_{0,N}$ into both spaces.
\end{Defthm}

\begin{Def}
  A \newword{boundary stratum} in $\Mbar_{0,N}(\be)$ is the image, under
  $\rho_{\be}$, of a boundary stratum in $\Mbar_{0,N}$.
\end{Def}
The homology group $H_{2k}(\Mbar_{0,N}(\be))$ is isomorphic to the
Chow group $A_k(\Mbar_{0,N}(\be))$ and is generated by the classes of
boundary strata (\cite{Ceyhan2009}).
\begin{rem}
  $\Mbar_{0,N}=\Mbar_{0,N}(\be)$ for $\be=(1,\ldots,1).$
\end{rem}

\begin{Def}
  Let $\be$ be a weight datum, and $\sigma$ be an $N$-marked stable
  tree. A vertex $v$ on $\sigma$ is called \newword{$\be$-stable} if
  \begin{align*}
    \sum_{\delta\in\Delta_v}\min\left\{1,\sum_{i|\delta=\delta(v\to
        i)}\epsilon_i\right\}>2.
  \end{align*}
\end{Def}
For $[C,p_1,\ldots,p_N]\in\Mbar_{0,N}$ with dual tree $\sigma$, an
$\be$-stable curve representing $\rho_{\be}([C])$ is obtained by
contracting to a point every irreducible component $C_v$ corresponding
to $v\in\cV(\sigma)$ that is not $\be$-stable. It follows that
$\sigma$ has at least one $\be$-stable vertex. Also, the image under
$\rho_{\be}$ of the boundary stratum $S_\sigma$ has dimension
\begin{align*}
  \sum_{\substack{v\in\cV(\sigma)\\\text{$v$ $\be$-stable}}}\md(v).
\end{align*}
Thus, we obtain:
\begin{lem}\label{lem:CriterionForKernel}
  Let $S_\sigma$ be a boundary stratum in $\Mbar_{0,N}$. The
  pushforward $(\rho_{\be})_*([S_\sigma])$ is nonzero if and only if
  every vertex $v\in\cV(\sigma)$ with positive moduli dimension is
  $\be$-stable.
\end{lem}

\section{Admissible covers}\label{sec:AdmissibleCovers}
There is a widely used compactification of the Hurwitz space $\H$ by a
space $\Hbar$ of \newword{admissible covers}. Moduli spaces of
admissible covers were constructed in \cite{HarrisMumford1982} by
Harris and Mumford, and parametrize finite maps between possibly nodal
curves. In general, they are only coarse moduli spaces, with quotient
singularities. For technical ease, we introduce a class of Hurwitz
spaces whose admissible covers compactifications are fine moduli
spaces. We refer to Hurwitz spaces in this class as \emph{fully
  marked}.

\subsection{Fully marked Hurwitz spaces}\label{sec:FullyMarked}
\begin{Def}\label{Def:FullyMarked}
  Given $(\bA,\bB,d,F,\br,\rm)$ as in Definition
  \ref{Def:HurwitzSpace} with Condition 2 strengthened to:
  \begin{itemize}
  \item (Condition $2'$) For all $b\in\bB,$ the multiset
    $(\rm(a))_{a\in F^{-1}(b)}$ is \textbf{equal to} $\br(b),$
  \end{itemize}
  we refer to $\cH(\bA,\bB,d,F,\br,\rm)$ as a \newword{fully marked
    Hurwitz space}.
\end{Def}

Given any Hurwitz space $\cH=\cH(\bA,\bB,d,F,\br,\rm)$, we can
construct a fully marked Hurwitz space $\Hfull$ with a finite covering
map $\nu:\Hfull\to\cH$ as follows.  We first construct a superset
$\Afull$ of $\bA$, extending the functions $F$ and $\rm.$ For every
$b\in\bB$, and for every $r$ in the multiset complement
$\br(b)\setminus(\rm(a))_{a\in F^{-1}(b)\cap\bA}$, add an element
$a(b,r)$ to $\Afull$, set $F(a(b,r))=b,$ and set $\rm(a(b,r))=r.$ The
data $(\Afull,\bB,d,F,\br,\rm)$ satisfy Conditions 1 and $2'$, so
$\Hfull=\H(\Afull,\bB,d,F,\br,\rm)$ is a fully marked Hurwitz space.

Let $\Aut(\Afull\setminus\bA)$ be the subgroup of permutations of
$\Afull\setminus\bA$ preserving the functions $F$ and $\rm$. This
automorphism group acts freely on $\Hfull$ by relabeling points in
$\Afull\setminus\bA,$ and the quotient of this action is $\H.$ Denote
by $\nu$ the quotient map $\Hfull\to\H.$ The fully marked Hurwitz
space $\Hfull$ admits a map $\pi_{\Afull}$ to $\M_{0,\Afull}$. Also,
the injection $\bA\into\Afull$ yields a forgetful map
$\mu:\M_{0,\Afull}\to\M_{0,\bA}.$ The following diagrams commutes:
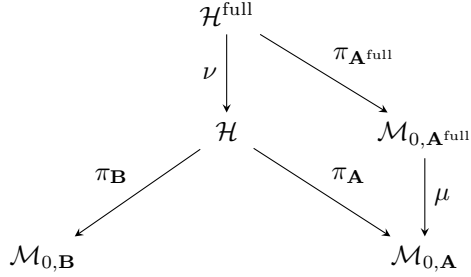
\begin{center}
  \begin{tikzpicture}
    \matrix(m)[matrix of math nodes,row sep=3em,column
    sep=4em,minimum width=2em] {
      &\Hfull&\\
      &\H&\M_{0,\Afull}\\
      \M_{0,\bB}&&\M_{0,\bA}\\};
    \path[-stealth] (m-1-2) edge node [left] {$\nu$} (m-2-2);
    \path[-stealth] (m-1-2) edge node [above right] {$\pi_{\Afull}$}
    (m-2-3);
    \path[-stealth] (m-2-3) edge node [right] {$\mu$} (m-3-3);
    \path[-stealth] (m-2-2) edge node [above right] {$\pi_{\bA}$} (m-3-3);
    \path[-stealth] (m-2-2) edge node [above left] {$\pi_{\bB}$} (m-3-1);
  \end{tikzpicture}
\end{center}
For $\Gamma$ any union of connected components of $\H,$
$\Gamma^{\full}:=\nu^{-1}(\Gamma)$ is a union of connected components
of $\H^{\full}.$ For compactifications $X_{\bB}$ and $X_{\bA}$ of
$\M_{0,\bB}$ and $\M_{0,\bA}$ respectively,
$(\Gamma,\pi_{\bB},\pi_{\bA})$ and
$(\Gamma^{\full},\pi_{\bB}\circ\nu,\mu\circ\pi^{\Afull})$ are both
Hurwitz correspondences from $X_{\bB}$ to $X_{\bA}$, and
$[\Gamma^{\full}]=(\deg\nu)[\Gamma]$ in $H_*(X_{\bB}\times X_{\bA}).$
This observation yields a useful lemma:
\begin{lem}\label{lem:ReduceToConnectedFullyMarked}
  Let $(\Gamma,\pi_\bB,\pi_{\bA}):X_\bB\ratt X_\bA$ be any Hurwitz
  correspondence. Then $$[\Gamma]=\frac{1}{\deg\nu}[\Gamma^{\full}],$$
  where each $\Gamma^{\full}$ is a union of connected components of a
  fully marked Hurwitz space $\Hfull$ corresponding to a superset
  $\Afull$ of $\bA,$ and $\nu:\Gamma^{\full}\to\Gamma$ is a finite
  covering map.
\end{lem}
This means that any Hurwitz correspondence can be written in terms of
(connected components of) fully marked Hurwitz spaces. These have
convenient compactifications by moduli spaces of admissible covers. We
use this Lemma in Proposition \ref{prop:KComposable}
and Theorem \ref{thm:PreservesFiltration}.

\subsection{Admissible Covers}\label{sec:DefAdmissibleCovers}
For an introduction see \cite{HarrisMorrison1998}, Chapter 3G.
\begin{Def}[\cite{HarrisMumford1982}]
  Fix $(\bA,\bB,d,F,\br,\rm)$ as in Definition
  \ref{Def:HurwitzSpace}. An $(\bA,\bB,d,F,\br,\rm)$-admissible cover
  is a map of curves $f:C\to D$, where
  \begin{enumerate}[(1)]
  \item $D$ is a $\bB$-marked genus zero stable curve,
  \item $C$ is a (not necessarily stable) connected nodal genus zero
    curve, with an injection from $\bA$ into the smooth locus of $C$,
  \item $f:C\to D$ is a finite map of degree $d$, such that
    \begin{itemize}
    \item for all $a\in\bA,$ $f(a)=F(a)$ (via the injections $\bA\into
      C$ and $\bB\into D$),
    \item for all $b\in\bB,$ the branching of $f$ over $b$ is given by
      the partition $\br(b),$
    \item for all $a\in\bA,$ the local degree of $f$ at $a$ is equal
      to $\rm(a)$, 
    \item $\eta$ is a node on $C$ if and only if $f(\eta)$ is a node
      on $D$,
    \item (balancing condition) for $\eta$ a node of $C$ between
      irreducible components $C_1$ and $C_2$ of $C,$ $f(C_1)$ and
      $f(C_2)$ are distinct components of $D$, and the local degree of
      $f|_{C_1}$ at $\eta$ is equal to the local degree of $f|_{C_2}$ at
      $\eta.$
    \end{itemize}
  \end{enumerate}
\end{Def}
\begin{rem}
  If $(\bA,\bB,d,F,\br,\rm)$ satisfies Conditions 1 and $2'$ as in
  Definition \ref{Def:FullyMarked}, then $C$ is an $\bA$-marked
  \emph{stable} curve.
\end{rem}
\begin{thm}[\cite{HarrisMumford1982}]\label{thm:AdmissibleCovers}
  Given $(\bA,\bB,d,F,\br,\rm)$ satisfying Conditions 1 and $2'$ as in
  Definition \ref{Def:FullyMarked}, there is a projective variety
  $\Hbar=\Hbar(\bA,\bB,d,F,\br,\rm)$ that is a fine moduli space for
  $(\bA,\bB,d,F,\br,\rm)$-admissible covers, and contains $\H$ as a
  dense open subset. $\Hbar$ extends the maps $\pi_\bA$ and $\pi_\bB$
  to maps $\bar{\pi_{\bA}}$ and $\bar{\pi_{\bB}}$ to $\Mbar_{0,\bA}$
  and $\Mbar_{0,\bB}$, respectively, with
  $\bar{\pi_\bB}:\Hbar\to\Mbar_{0,\bB}$ a finite flat map. $\Hbar$ may
  not be normal, but its normalization is smooth.
\end{thm}
\begin{rem}
  The irreducible components of $\Hbar$ correspond to the connected
  components of $\H.$
\end{rem}
\begin{rem}
  The construction in \cite{HarrisMumford1982} is very general, but as
  stated allows for only simple ramification and no marked points on
  the source curve, so does not apply to our case. However, it is
  easily modified: consider $\Mbar_{0,\bA}\times\Mbar_{0,\bB}$ with
  its two universal curves $\Ubar_{0,\bA}$ and $\Ubar_{0,\bB}$. Let
  $\Hilb$ be the relative Hilbert scheme of degree $d$ morphisms
  $\Ubar_{0,\bA}\to\Ubar_{0,\bB}$. Then the locus
  $$\Hbar:=\{f:C\to D|\mbox{$f$ is an
    $(\bA,\bB,d,F,\br,\rm)$-admissible cover}\}$$
  is a closed subscheme of $\Hilb$ by the proof of Theorem 4 in
  \cite{HarrisMumford1982}.
\end{rem}
Given $(\bA,\bB,d,F,\br,\rm)$ satisfying Conditions 1 and 2 as in
Definition \ref{Def:HurwitzSpace}, but not $2'$ as in Definition
\ref{Def:FullyMarked}, there is still a compactification $\Hbar$ of
$\H=\H(\bA,\bB,d,F,\br,\rm)$ by admissible covers. Consider the
corresponding fully marked Hurwitz space $\Hfull$ and its admissible
covers compactification $\Hbarfull$. The action of
$\Aut(\Afull\setminus\bA)$ on $\Hfull$ extends to an action on
$\Hbarfull$, but this action is no longer free, so the quotient
$\Hbar$ has \textit{orbifold singularities}, and is only a coarse
moduli space.
\subsection{Stratification of
  $\Hbar$}\label{sec:StratificationOfAdmissibleCovers}
Moduli spaces of admissible covers have a stratification analogous to
and compatible with that of $\Mbar_{0,N}$.  In this section we fix
$(\bA,\bB,d,F,\br,\rm)$ satisfying Conditions 1 and $2'$ as in
Definition \ref{Def:FullyMarked}. Let $\H=\H(\bA,\bB,d,F,\br,\rm)$ be
the corresponding fully marked Hurwitz space, and let
$\Hbar=\Hbar(\bA,\bB,d,F,\br,\rm)$.
\begin{Def}
  Given any admissible cover $[f:C\to D]\in\Hbar,$ we can extract its
  \newword{combinatorial type}
  $\gamma=(\sigma,\tau,d^{\vert},f^{\vert},(F_v)_{v\in\mathcal{V}(\sigma)},(\br_v)_{v\in\mathcal{V}(\sigma)},(\rm_v)_{v\in\mathcal{V}(\sigma)}),$
  where:
  \begin{enumerate}[(1)]
  \item $\sigma$ is the dual tree of $C$,
  \item $\tau$ is the dual tree of $D$,
  \item $d^{\vert}:\mathcal{V}(\sigma)\to\Z^{>0}$ sends $v$ to
    $\deg f|_{C_v}$,
  \item $f^{\vert}:\mathcal{V}(\sigma)\to\mathcal{V}(\tau)$, where for
    $C_v$ an irreducible component of $C$ mapping under $f$ to $D_w$
    an irreducible component of $D,$ $f^{\vert}$ sends $v$ to $w$,
  \item For $v\in\mathcal{V}(\sigma)$, the map
    $F_v:\Delta_v\to\Delta_{f^{\vert}(v)}$ is given by the restriction
    of $f$ via the inclusions $\Delta_v\into C_v$ and
    $\Delta_{f^{\vert}(v)}\into D_{f^{\vert}(v)}$,
  \item For $v\in\mathcal{V}(\sigma),$ the map
    $\br_v:\Delta_{f^{\vert}(v)}\to\{\mbox{partitions of
      $d^{\vert}(v)$}\}$ sends $\delta'\in\Delta_{f^{\vert}(v)}$ to
    the branching of $f|_{C_v}:C_v\to D_{f^{\vert}(v)}$ over
    $\delta',$ via the inclusion $\Delta_{f^{\vert}(v)}\into
    D_{f^{\vert}(v)}$, and
  \item For $v\in\mathcal{V}(\sigma)$, the map
    $\rm_v:\Delta_v\to\Z^{>0}$ sends $\delta\in\Delta_v$ to the local
    degree of $f|_{C_v}$ at $\delta,$ via the inclusion
    $\Delta_v\into C_v.$
  \end{enumerate}
\end{Def}
\begin{Def}
  The closure $G_\gamma$ of $\{f':C'\to D'|\mbox{$f'$ has
    combinatorial type $\gamma$}\}$ is a subvariety of $\Hbar$.  We
  call such a subvariety a \newword{boundary stratum} of $\Hbar$.
\end{Def}
The boundary stratum $G_\gamma$ in $\Hbar$ can be decomposed into a
product of lower-dimensional spaces of admissible covers. For
$v\in\mathcal{V}(\sigma),$ the data
$(\Delta_v,\Delta_{f^{\vert}(v)},d^\vert(v),F_v,\br_v,\rm_v)$ satisfy
Conditions 1 and $2'$ as in Definition \ref{Def:FullyMarked}. Denote
by $\Hbar_v$ the corresponding space of admissible covers. The space
$\Hbar_v$ admits maps to the moduli space $\Mbar_{0,\Delta_v}$ of
source curves and the moduli space $\Mbar_{0,\Delta_{f^{\vert}(v)}}$
of target curves. For $w\in\mathcal{V}(\tau),$ set $\Hbar_w$ to be the
\emph{fibered} product
\begin{align*}
  \Hbar_w:=\prod_{v\in(f^{\vert})^{-1}(w)}\Hbar_v,
\end{align*}
where the product is fibered over the common moduli space
$\Mbar_{0,\Delta_w}$ of target curves.

The fibered product $\Hbar_w$ is itself a moduli space of possibly
disconnected admissible covers, admitting a map
$\bar\pi_w^{\mathrm{source}}$ to the moduli space
$\prod_{v\in(f^{\vert})^{-1}(w)}\Mbar_{0,\Delta_v}$ of source curves
and a finite flat map $\bar\pi_w^{\mathrm{target}}$ to the moduli
space $\Mbar_{0,\Delta_w}$ of target curves. The stratum $G_\gamma$ is
isomorphic to $\prod_{w\in\mathcal{V}(\tau)}\Hbar_w.$

Recall that the boundary stratum $T_\tau$ in $\Mbar_{0,\bB}$ is
isomorphic to $\prod_{w\in\mathcal{V}(\tau)}\Mbar_{0,\Delta_w}.$ We
have
\begin{center}
  \begin{tikzpicture}
    \matrix(m)[matrix of math nodes,row sep=3em,column sep=8em,minimum
    width=2em] {
      \prod_{w\in\mathcal{V}(\tau)}\Hbar_w&\prod_{w\in\mathcal{V}(\tau)}\Mbar_{0,\Delta_w}\\
      G_\gamma&T_\tau\\
      \Hbar&\Mbar_{0,\bB}\\
    };
    \path[-stealth] (m-1-1) edge node [above] {$\prod_{w\in\mathcal{V}(\tau)}\bar\pi_w^{\mathrm{target}}$} (m-1-2);
    \path[-stealth] (m-2-1) edge node [above] {$\bar{\pi_\bB}|_{G_\gamma}$} (m-2-2);
    \path[-stealth] (m-3-1) edge node [above] {$\bar{\pi_\bB}$} (m-3-2);
    \path[-stealth,<->] (m-1-1) edge node [left] {$\cong$} (m-2-1);
    \path[-stealth,<->] (m-1-2) edge node [right] {$\cong$} (m-2-2);
    \path[-stealth,right hook->] (m-2-1) edge  (m-3-1);
    \path[-stealth,right hook->] (m-2-2) edge  (m-3-2);
  \end{tikzpicture}
\end{center}
All rightwards arrows are finite flat maps to moduli spaces of target
curves.

Similarly, the stratum $S_\sigma$ in $\Mbar_{0,\bA}$ is isomorphic to
$\prod_{v\in\mathcal{V}(\sigma)}\Mbar_{0,\Delta_v},$ and we have
\begin{center}
  \begin{tikzpicture}
    \matrix(m)[matrix of math nodes,row sep=2em,column sep=8em,minimum
    width=2em] {
      \prod\limits_{w\in\mathcal{V}(\tau)}\Hbar_w&\prod\limits_{w\in\mathcal{V}(\tau)}\prod\limits_{v\in
        (f^{\vert})^{-1}(w)}\Mbar_{0,\Delta_v}\\
      &\prod\limits_{v\in\mathcal{V}(\sigma)}\Mbar_{0,\Delta_v}\\
      G_\gamma&S_\sigma\\
      \Hbar&\Mbar_{0,\bA}\\
    }; \path[-stealth] (m-1-1) edge node [above] {$\prod_{w\in\mathcal{V}(\tau)}\bar\pi_w^{\mathrm{source}}$} (m-1-2);
    \path[-stealth] (m-3-1) edge node [above] {$\bar{\pi_\bA}|_{G_\gamma}$} (m-3-2);
    \path[-stealth] (m-4-1) edge node [above] {$\bar{\pi_\bA}$} (m-4-2);
    \path[-stealth,<->] (m-1-1) edge node [left] {$\cong$} (m-3-1);
    \draw[double distance=2pt,shorten <=5pt,shorten >=5pt] (m-1-2) -- (m-2-2);
    \path[-stealth,<->] (m-2-2) edge node [right] {$\cong$} (m-3-2);
    \path[-stealth,right hook->] (m-3-1) edge  (m-4-1);
    \path[-stealth,right hook->] (m-3-2) edge  (m-4-2);
  \end{tikzpicture}
\end{center}
All rightwards arrows are maps to moduli spaces of source curves.

$G_\gamma$ is not necessarily irreducible. Its irreducible components
are of the form $$J=\prod_{w\in\mathcal{V}(\tau)}\bar{\cJ_w},$$ where
$\bar{\cJ_w}$ is an irreducible component of $\Hbar_w.$

\subsection{Cohomology and Chow groups}\label{sec:CohomologyChow} Let
$\Hbar=\Hbar(\bA,\bB,d,F,\br,\rm)$ be as in the statement of Theorem
\ref{thm:AdmissibleCovers}. Although $\Hbar$ is not smooth, there is a
pullback map of Chow groups
$(\bar{\pi_{\bB}})^*:A_k(\Mbar_{0,\bB})\to A_k(\Hbar)$
(\cite{Fulton1998}, Section 6.2). We thus have
$$(\bar{\pi_{\bA}})_*\circ(\bar{\pi_{\bB}})^*:A_k(\Mbar_{0,\bB})\to
A_k(\Mbar_{0,\bA}).$$
For any resolution of singularities $\chi:\tilde{\H}\to\Hbar,$
$\chi_*\circ(\bar{\pi_{\bB}}\circ\chi)^*:A_k(\Mbar_{0,\bB})\to
A_k(\Hbar)$
agrees with $(\bar{\pi_{\bB}})^*$ (\cite{Fulton1998},
Theorem 6.2a). Thus $(\bar{\pi_{\bA}})_*\circ(\bar{\pi_{\bB}})^*$
agrees with
$(\bar{\pi_{\bA}}\circ\chi)_*\circ(\bar{\pi_{\bB}}\circ\chi)^*$ as
maps from $A_k(\Mbar_{0,\bB})$ to $A_k(\Mbar_{0,\bA}).$

By \cite{Keel1992}, the `cycle class' map gives canonical isomorphisms
$A_k(\Mbar_{0,N})\to H_{2k}(\Mbar_{0,N},\Z).$ Under these isomorphisms,
$$(\bar{\pi_{\bA}}\circ\chi)_*\circ(\bar{\pi_{\bB}}\circ\chi)^*:A_k(\Mbar_{0,\bB})\to
A_k(\Mbar_{0,\bA})$$
is identified
with
$$(\bar{\pi_{\bA}}\circ\chi)_*\circ(\bar{\pi_{\bB}}\circ\chi)^*:H_{2k}(\Mbar_{0,\bB},\Z)\to
H_{2k}(\Mbar_{0,\bA},\Z).$$ Thus the pushforward
$[\H]_*:H_{2k}(\Mbar_{0,\bB},\Z)\to H_{2k}(\Mbar_{0,\bA},\Z)$ may be
identified with
$(\bar{\pi_{\bA}})_*\circ(\bar{\pi_{\bB}})^*:A_k(\Mbar_{0,\bB})\to
A_k(\Mbar_{0,\bA}).$ An analogous statement holds for $[\Gamma]_*,$
where $\Gamma$ is any union of connected components of $\H.$ We use
this identification in Proposition \ref{prop:KComposable} and Theorem
\ref{thm:PreservesFiltration}.

\section{Hurwitz correspondences are algebraically stable on
  $\Mbar_{0,N}$}\label{sec:HurwitzStableCurves}
The admissible covers compactification $\Hbar$ of
$\H=\H(\bA,\bB,d,F,\br,\rm)$ naturally lives over the spaces
$\Mbar_{0,N}$ of stable curves: it extends the ``target curve'' map to
$\Mbar_{0,\bB}$ and the ``source curve'' map to $\Mbar_{0,\bA}$. We
treat the Hurwitz correspondence $\H$ as a multivalued map from
$\M_{0,\bB}$ to $\M_{0,\bA}$. More precisely, $\H$ induces a map from
$\M_{0,\bB}$ to $\Sym^{(\deg\pi_{\bB})}\M_{0,\bA}.$ Since
$\bar{\pi_{\bB}}:\Hbar\to\Mbar_{0,\bB}$ is finite and flat, this
extends to a map from $\Mbar_{0,\bB}$ to
$\Sym^{(\deg\pi_{\bB})}\Mbar_{0,\bA}.$ Thus the rational
correspondence
$(\H,\pi_{\bB},\pi_{\bA}):\Mbar_{0,\bB}\ratt\Mbar_{0,\bA}$ may be
treated as a regular correspondence. This is supported by the fact
that Hurwitz correspondences are homologically composable on the
stable curves spaces $\Mbar_{0,N}$. The following proposition is the
result of a collaboration with Sarah Koch and David Speyer.
\begin{prop}[Koch, Ramadas, Speyer]\label{prop:KComposable}
  Let $(\Gamma,\pi_1,\pi_2):\Mbar_{0,N_1}\ratt\Mbar_{0,N_2}$ and
  $(\Gamma',\pi_2',\pi_3'):\Mbar_{0,N_2}\ratt\Mbar_{0,N_3}$ be Hurwitz
  correspondences. Then for all $k$,
  \begin{align*}
    [\Gamma'\circ\Gamma]_*=[\Gamma']_*\circ[\Gamma]_*
  \end{align*}
  as maps from $H_{2k}(\Mbar_{0,N_1})$ to $H_{2k}(\Mbar_{0,N_3}).$
\end{prop}
\begin{proof}
  We prove instead that
  $[\Gamma'\circ\Gamma]_*=[\Gamma']_*\circ[\Gamma]_*$ as maps from
  $A_k(\Mbar_{0,N_1})$ to $A_k(\Mbar_{0,N_3})$ (see Section
  \ref{sec:CohomologyChow}).

  By Lemma \ref{lem:ReduceToConnectedFullyMarked}, we may reduce to
  the case in which $\Gamma$ and $\Gamma'$ are unions of connected
  components of fully marked Hurwitz spaces $\H$ and $\H'$
  respectively. The maps $\pi_1$ and $\pi_2'$ to the moduli spaces of
  target curves $\M_{0,N_1}$ and $\M_{0,N_2}$ are finite covering
  maps. Also we have $\pi_2(\Gamma)\subseteq\M_{0,N_2}$. Applying the
  definition of composition of rational correspondences given in
  Section \ref{sec:RationalCorrespondences}, we may
  set $$\Gamma'\circ\Gamma=\Gamma\times_{\M_{0,N_2}}\Gamma',$$ with
  maps $\pr$ and $\pr'$ to $\Gamma$ and $\Gamma'$ respectively. The
  map $\pr$ is the pullback of a covering map, so is itself a covering
  map.  Let $\bar{\Gamma}$ and $\bar{\Gamma'}$ be the closures of
  $\Gamma$ and $\Gamma'$ in the admissible covers spaces $\Hbar$ and
  $\Hbar'.$ We have the diagram
  
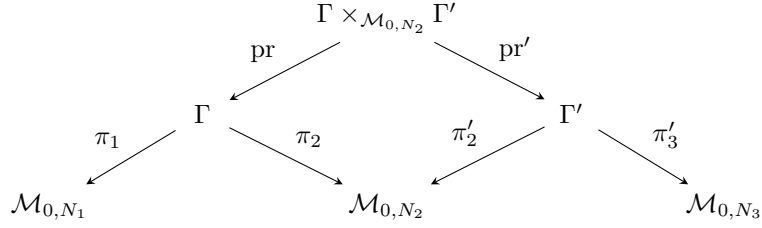
\begin{center}
    \begin{tikzpicture}
      \matrix(m)[matrix of math nodes,row sep=2em,column
      sep=3em,minimum width=2em] {
        &&\Gamma\times_{\M_{0,N_2}}\Gamma'&&\\
        &\Gamma&&\Gamma'&\\
        \M_{0,N_1}&&\M_{0,N_2}&&\M_{0,N_3}\\
      };
      \path[-stealth] (m-1-3) edge node [above left] {$\pr$} (m-2-2);
      \path[-stealth] (m-1-3) edge node [above right] {$\pr'$} (m-2-4);
      \path[-stealth] (m-2-2) edge node [above left] {$\pi_1$} (m-3-1);
      \path[-stealth] (m-2-2) edge node [above right] {$\pi_2$} (m-3-3);
      \path[-stealth] (m-2-4) edge node [above left] {$\pi_2'$} (m-3-3);
      \path[-stealth] (m-2-4) edge node [above right] {$\pi_3'$} (m-3-5);
    \end{tikzpicture}
  \end{center}
  
  We also have the diagram of
  compactifications:
  
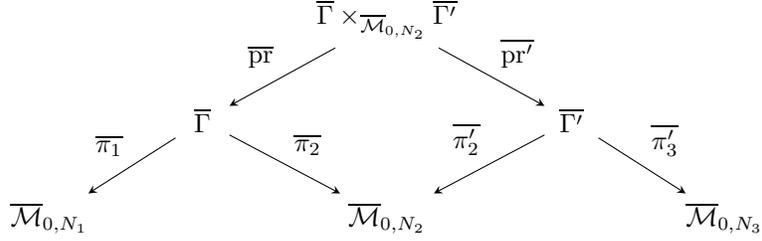
\begin{center}
    \begin{tikzpicture}
      \matrix(m)[matrix of math nodes,row sep=2em,column
      sep=3em,minimum width=2em] {
        &&\bar{\Gamma}\times_{\Mbar_{0,N_2}}\bar{\Gamma'}&&\\
        &\bar{\Gamma}&&\bar{\Gamma'}&\\
        \Mbar_{0,N_1}&&\Mbar_{0,N_2}&&\Mbar_{0,N_3}\\
      };
      \path[-stealth] (m-1-3) edge node [above left] {$\bar{\pr}$} (m-2-2);
      \path[-stealth] (m-1-3) edge node [above right] {$\bar{\pr'}$} (m-2-4);
      \path[-stealth] (m-2-2) edge node [above left] {$\bar{\pi_1}$} (m-3-1);
      \path[-stealth] (m-2-2) edge node [above right] {$\bar{\pi_2}$} (m-3-3);
      \path[-stealth] (m-2-4) edge node [above left] {$\bar{\pi_2'}$} (m-3-3);
      \path[-stealth] (m-2-4) edge node [above right] {$\bar{\pi_3'}$}
      (m-3-5);
    \end{tikzpicture}
  \end{center}

  Since $\bar{\pi_1}$ and $\bar{\pi_2'}$ are finite and flat, so is
  $\bar{\pi_1}\circ\bar{\pr}.$ This means that
  $\bar{\Gamma}\times_{\Mbar_{0,N_2}}\bar{\Gamma'}$ has no irreducible
  components supported over the boundary
  $\Mbar_{0,N1}\setminus\M_{0,N_1}.$ There is an inclusion
  $\Gamma\times_{\M_{0,N_2}}\Gamma'\into\bar{\Gamma}\times_{\Mbar_{0,N_2}}\bar{\Gamma'}$
  whose image is $(\bar{\pi_1}\circ\bar{\pr})^{-1}(\M_{0,N_1}).$ By
  the above this is an inclusion as a dense open set. We therefore
  have
  \begin{align*}
    [\Gamma'\circ\Gamma]_*&=[\bar\Gamma\times_{\Mbar_{0,N_2}}\bar{\Gamma'}]_*\\
    &=(\bar{\pi_3'}\circ\bar{\pr'})_*\circ(\bar{\pi_1}\circ\bar{\pr})^*\\
    &=(\bar{\pi_3'})_*\circ(\bar{\pr'})_*\circ(\bar{\pr})^*\circ(\bar{\pi_1})^*.
  \end{align*}
  On the other hand,
  \begin{align*}
    [\Gamma']_*\circ[\Gamma]_*&=[\bar{\Gamma'}]_*\circ[\bar\Gamma]_*\\
    &=(\bar{\pi_3'})_*\circ(\bar{\pi_2'})^*\circ(\bar{\pi_2})_*\circ(\bar{\pi_1})^*\\
    &=(\bar{\pi_3'})_*\circ(\bar{\pr'})_*\circ(\bar{\pr})^*\circ(\bar{\pi_1})^*\\
    &=[\Gamma'\circ\Gamma]_*.
  \end{align*}
  Here, the third equality follows from the fact (Proposition 1.7 in
  \cite{Fulton1998}) that for any fibered square of varieties
  
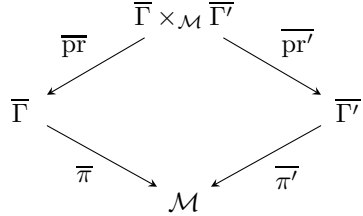
\begin{center}
    \begin{tikzpicture}
      \matrix(m)[matrix of math nodes,row sep=2em,column
      sep=3em,minimum width=2em] {
        &\bar{\Gamma}\times_{\M}\bar{\Gamma'}&\\
        \bar{\Gamma}&&\bar{\Gamma'}\\
        &\M&\\
      };
      \path[-stealth] (m-1-2) edge node [above left] {$\bar{\pr}$} (m-2-1);
      \path[-stealth] (m-1-2) edge node [above right] {$\bar{\pr'}$} (m-2-3);
      \path[-stealth] (m-2-1) edge node [below left] {$\bar{\pi}$} (m-3-2);
      \path[-stealth] (m-2-3) edge node [below right] {$\bar{\pi'}$} (m-3-2);
    \end{tikzpicture}
  \end{center}
  where $\bar{\pi'}$ is a flat map and $\bar{\pi}$ is proper, we have
  \begin{align*}
    (\bar{\pi'})^*\circ(\bar{\pi})_*&=(\bar{\pr'})_*\circ(\bar{\pr})^*. \qedhere
  \end{align*}
\end{proof}

By duality of pushforward and pullback, and the fact (Keel,
\cite{Keel1992}) that $H^{2k}(\Mbar_{0,N})=H^{k,k}(\Mbar_{0,N}),$ we obtain:  
\begin{cor}[Koch, Ramadas, Speyer]\label{cor:LargestEigenvalue}
  Let $(\Gamma,\pi_1,\pi_2):\Mbar_{0,N}\ratt\Mbar_{0,N}$ be a dominant
  Hurwitz self-correspondence. Then $\Gamma$ is algebraically stable,
  and its $k$th dynamical degree is the absolute value of the dominant
  eigenvalue of $[\Gamma]_*:H_{2k}(\Mbar_{0,N})\to
  H_{2k}(\Mbar_{0,N})$
\end{cor}

\begin{rem}\label{rem:EffectiveCone}
  Let $(\Gamma,\pi_1,\pi_2):\M_{0,N}\ratt\M_{0,N}$ be as above, and
  let $\bar{\Gamma}$ be the admissible covers compactification of
  $\Gamma$, with its maps $\bar{\pi_1}$ and $\bar{\pi_2}$ to
  $\Mbar_{0,N}.$ The map $\bar{\pi_1}$ is flat, so $(\bar{\pi_1})^*$
  takes effective classes on $\Mbar_{0,N}$ to effective classes on
  $\bar{\Gamma}.$ Pushforwards always preserve effectiveness, so
  $[\Gamma]_*=(\bar{\pi_2})_*\circ(\bar{\pi_1})^*:H_{2k}(\Mbar_{0,N})\to
  H_{2k}(\Mbar_{0,N})$ preserves the cone of effective classes. By
  continuity, $[\Gamma]_*$ preserves the \emph{pseudoeffective cone},
  namely the closure of the cone of effective classes. The
  pseudoeffective cone of any projective variety is a closed cone with
  nonempty interior that contains no lines
  (\cite{BoucksomFavreJonsson2009,FulgerLehmann2014}). It follows from
  the theory of cone-preserving operators (\cite{SchneiderTam2007})
  that $[\Gamma]_*$ has a nonnegative real dominant eigenvalue, with a
  pseudoeffective eigenvector.
\end{rem}

\section{Hurwitz correspondences preserve a natural filtration of
  $H_{2k}(\Mbar_{0,N})$}\label{sec:Filtration}
Since Hurwitz correspondences are algebraically stable on
$\Mbar_{0,N}$, their dynamical degrees are dominant eigenvalues of
pushforward maps induced on the homology groups of $\Mbar_{0,N}$
(Corollary \ref{cor:LargestEigenvalue}). Here, we describe these
pushforward maps: we show that they preserve a natural combinatorially
defined filtration. We recall the notation and terminology introduced
in Section \ref{sec:StableCurves}.
\begin{Def}
  Let $S_\sigma$ be a $k$-dimensional boundary stratum in
  $\Mbar_{0,N}$. Then $\sum_{v\in\mathcal{V}(\sigma)}\md(v)=k$. Set
  $\lambda_\sigma$ to be the multiset
  \begin{align*}
    (\md(v))_{\text{$v\in\cV(\sigma)$ with $\md(v)\ne0$}}.
  \end{align*}
  Then $\lambda_\sigma$ is a partition of $k$; we say $\lambda_\sigma$
  is the partition \newword{induced by} the stratum $S_\sigma$.
\end{Def}
$S_\sigma$ is isomorphic to
$\prod_{v\in\mathcal{V}(\sigma)}\Mbar_{0,\Delta_v}$. Each factor
$\Mbar_{0,\Delta_v}$ has dimension $\md(v)$, so 
$\lambda_\sigma$ encodes the nonzero dimensions of the factors when we
write $S_\sigma$ as this product.

\begin{Def}
  For a partition $\lambda$ of $k$ let $\Lambda_N^{\le\lambda}$ be the
  subspace of $H_{2k}(\Mbar_{0,N})$ generated by
  $k$-dimensional boundary strata that induce the partition $\lambda$
  or any refinement. For $\bP$ a finite set we define
  $\Lambda_{\bP}^{\le\lambda}\subseteq H_{2k}(\Mbar_{0,\bP})$
  analogously.
\end{Def}
If $\lambda_1$ is a refinement of
$\lambda_2,$ we write $\lambda_1\le\lambda_2$. This is a partial
ordering on the set of partitions of $k$. Clearly if
$\lambda_1\le\lambda_2$, then $\Lambda_N^{\le\lambda_1}\subseteq
\Lambda_N^{\le\lambda_2}.$

The dual tree of a $k$-dimensional boundary stratum $S_\sigma$ in
$\Mbar_{0,N}$ has $N-k-2$ vertices. This means that the partition
$\lambda_\sigma$ has at most $N-k-2$ parts. Conversely, any partition
of $k$ with at most $N-k-2$ parts arises from a boundary stratum in
$\Mbar_{0,N}$. We say that $\lambda$ is \newword{realizable} on
$\M_{0,N}$ if $\lambda$ has at most $N-k-2$ parts.

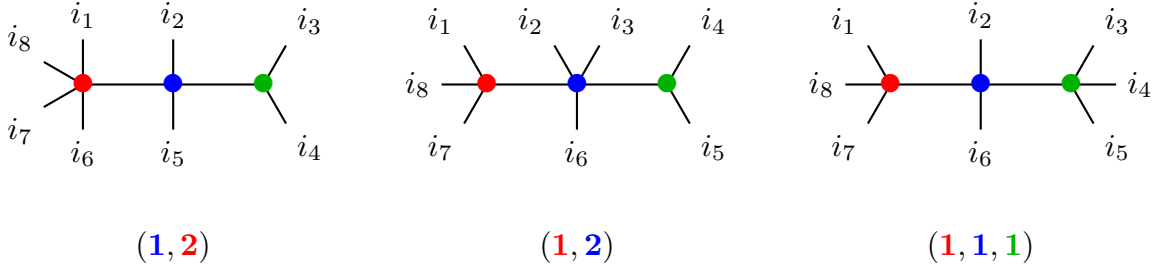
\begin{figure}\centering

  \begin{tikzpicture}[scale=1.2, every node/.style={transform shape}]
    \draw[thick] (0,0)--(1,0); \draw[thick] (1,0)--(2,0); \draw[thick]
    (1,0)-- ++(90:.5); \draw[thick] (1,0)-- ++(270:.5); \draw[thick]
    (0,0)-- ++(90:.5); \draw[thick] (0,0)-- ++(150:.5); \draw[thick]
    (0,0)-- ++(210:.5); \draw[thick] (0,0)-- ++(270:.5); \draw[thick]
    (2,0)-- ++(60:.5); \draw[thick] (2,0)-- ++(300:.5);

    \draw[red,scale=1.5] (0,0) node {$\bullet$}; \draw[blue,scale=1.5]
    (2/3,0) node {$\bullet$}; \draw[darkgreen,scale=1.5] (4/3,0) node
    {$\bullet$};

    \draw (1,0)++(90:.5) node [above] {$i_2$};
    \draw (1,0)++(270:.5) node [below] {$i_5$};
    \draw (0,0)++(90:.5) node [above] {$i_1$};
    \draw (0,0)++(150:.5) node [above left] {$i_8$};
    \draw (0,0)++(210:.5) node [below left] {$i_7$};
    \draw (0,0)++(270:.5) node [below] {$i_6$};
    \draw (2,0)++(60:.5) node [above right] {$i_3$};
    \draw (2,0)++(300:.5) node [below right] {$i_4$};

    \draw (1,-1.5) node [below]
    {$(\textcolor{blue}{\mathbf{1}},\textcolor{red}{\mathbf{2}})$};
  \end{tikzpicture}
  \quad\quad
  \begin{tikzpicture}[scale=1.2, every node/.style={transform shape}]
    \draw[thick] (0,0)--(1,0); \draw[thick] (1,0)--(2,0); \draw[thick]
    (1,0)-- ++(60:.5); \draw[thick] (1,0)-- ++(120:.5); \draw[thick]
    (1,0)-- ++(270:.5); \draw[thick] (0,0)-- ++(120:.5); \draw[thick]
    (0,0)-- ++(180:.5); \draw[thick] (0,0)-- ++(240:.5); \draw[thick]
    (2,0)-- ++(60:.5); \draw[thick] (2,0)-- ++(300:.5);
            
    \draw[red,scale=1.5] (0,0) node {$\bullet$}; \draw[blue,scale=1.5]
    (2/3,0) node {$\bullet$}; \draw[darkgreen,scale=1.5] (4/3,0) node
    {$\bullet$};

    \draw (1,0)++(60:.5) node [above right] {$i_3$};
    \draw (1,0)++(120:.5) node [above left] {$i_2$};
    \draw (1,0)++(270:.5) node [below] {$i_6$};
    \draw (0,0)++(120:.5) node [above left] {$i_1$};
    \draw (0,0)++(180:.5) node [left] {$i_8$};
    \draw (0,0)++(240:.5) node [below left] {$i_7$};
    \draw (2,0)++(60:.5) node [above right] {$i_4$};
    \draw (2,0)++(300:.5) node [below right] {$i_5$};

    \draw (1,-1.5) node [below]
    {$(\textcolor{red}{\mathbf{1}},\textcolor{blue}{\mathbf{2}})$};
  \end{tikzpicture}
  \quad\quad
  \begin{tikzpicture}[scale=1.2, every node/.style={transform shape}]
    \draw[thick] (0,0)--(1,0); \draw[thick] (1,0)--(2,0); \draw[thick]
    (1,0)-- ++(90:.5); \draw[thick] (1,0)-- ++(270:.5); \draw[thick]
    (0,0)-- ++(120:.5); \draw[thick] (0,0)-- ++(180:.5); \draw[thick]
    (0,0)-- ++(240:.5); \draw[thick] (2,0)-- ++(60:.5); \draw[thick]
    (2,0)-- ++(0:.5); \draw[thick] (2,0)-- ++(300:.5);
            
    \draw[red,scale=1.5] (0,0) node {$\bullet$}; \draw[blue,scale=1.5]
    (2/3,0) node {$\bullet$}; \draw[darkgreen,scale=1.5] (4/3,0) node
    {$\bullet$};

    \draw (1,0)++(90:.5) node [above] {$i_2$};
    \draw (1,0)++(270:.5) node [below] {$i_6$};
    \draw (0,0)++(120:.5) node [above left] {$i_1$};
    \draw (0,0)++(180:.5) node [left] {$i_8$};
    \draw (0,0)++(240:.5) node [below left] {$i_7$};
    \draw (2,0)++(60:.5) node [above right] {$i_3$};
    \draw (2,0)++(300:.5) node [below right] {$i_5$};
    \draw (2,0)++(0:.5) node [right] {$i_4$};

    \draw (1,-1.5) node [below]
    {$(\textcolor{red}{\mathbf{1}},\textcolor{blue}{\mathbf{1}},\textcolor{darkgreen}{\mathbf{1}})$};
  \end{tikzpicture}
  \caption{Dual trees of boundary strata generating the subspace
    $\Lambda_8^{\le(1,2)}\subseteq H_6(\Mbar_{0,8})$}
  \label{fig:Filtration}
\end{figure}

\begin{lem}\label{lem:KontsevichManin}
  Let $S_{\sigma_0}$ be a $k$-dimensional boundary stratum inducing
  partition $\lambda_{\sigma_0}$ of $k$. Then there does not exist an
  equality in $H_{2k}(\Mbar_{0,N})$ of the form
  \begin{align*}
    [S_{\sigma_0}]=\sum_j\beta_j[S_{\sigma(j)}],
  \end{align*}
where every partition
  $\lambda_{\sigma(j)}$ is distinct from $\lambda_{\sigma_0}.$
\end{lem}
\begin{proof}
  We recall the additive relations $R(\sigma,v,i_1,\ldots,i_4)$
  introduced in Section \ref{sec:StableCurves} (Equation
  \ref{eq:KontsevichManin}). We can rewrite
  \begin{align*}
    R(\sigma,v,i_1,\ldots,i_4)&=\sum_{(\Delta_1,\Delta_2)}([S(i_1i_2\Delta_1|\Delta_2i_3i_4)]-[S(i_1i_3\Delta_1|\Delta_2i_2i_4)]).
  \end{align*}
  It is clear from definitions that $S(i_1i_2\Delta_1|\Delta_2i_3i_4)$
  and $S(i_1i_3\Delta_1|\Delta_2i_2i_4)$ induce the same partition of
  $k$, and they appear in $R(\sigma,v,i_1,\ldots,i_4)$ paired up and
  with opposite signs. Therefore for any partition $\lambda$ of $k$,
  the coefficients in $R(\sigma,v,i_1,\ldots,i_4)$ of boundary strata
  inducing $\lambda$ sum to zero. Since the relations
  $R(\sigma,v,i_1,\ldots,i_4)$ generate all relations among boundary
  strata, we conclude that for any additive relation $R$, the
  coefficients of boundary strata inducing a fixed partition $\lambda$
  sum to zero. Therefore there is no equality
  $[S_{\sigma_0}]=\sum_j\beta_j[S_{\sigma(j)}],$ where each partition
  $\lambda_{\sigma(j)}$ is different from $\lambda_{\sigma_0}$.
\end{proof}
We deduce:
\begin{cor}\label{cor:FiltrationNontrivial}
  \begin{enumerate}[(i)]
  \item For $S_\sigma$ a $k$-dimensional boundary stratum and
    $\lambda$ a partition of $k$, $S_\sigma\in\Lambda_N^{\le\lambda}$
    if and only if $\lambda_\sigma\le\lambda$.
  \item For $\lambda_1$ and $\lambda_2$ distinct realizable partitions of
    $k$, $\Lambda_N^{\le\lambda_1}\ne \Lambda_N^{\le\lambda_2}$.
  \end{enumerate}
\end{cor}
The collection of subspaces $\{\Lambda_N^{\le\lambda}\}_\lambda$ is a
poset-filtration for $H_{2k}(\Mbar_{0,N})$ indexed by
$\{\mbox{partitions of $k$}\}$ with the refinement partial
ordering. Denote by $(k)$ the one-part partition of $k$. Then
$\Lambda_N^{\le(k)}=H_{2k}(\Mbar_{0,N})$.

Forgetful maps respect the filtration $\{\Lambda_N^{\le\lambda}\}_\lambda$:
\begin{lem}\label{lem:ForgetfulMapPreservesFiltration}
  Let $N\ge N'\ge3.$ Then, if $\mu:\Mbar_{0,N}\to\Mbar_{0,N'}$ is the
  forgetful map sending $(C,p_1,\ldots,p_N)$ to
  $(C,p_1,\ldots,p_{N'})$, the pushforward
  $\mu_*:H_{2k}(\Mbar_{0,N})\to H_{2k}(\Mbar_{0,N'})$ sends
  $\Lambda_N^{\le\lambda}$ to $\Lambda_{N'}^{\le\lambda}$ for all
  partitions $\lambda$ of $k$.
\end{lem}
\begin{proof}
  We may assume $N'=N-1.$ Let $S_\sigma$ be any $k$-dimensional
  boundary stratum in $\Mbar_{0,N},$ inducing partition
  $\lambda_\sigma$ of $k$. We show
  $\mu_*([S_\sigma])\in\Lambda_{N'}^{\le\lambda_\sigma}.$ The image
  $\mu(S_\sigma)=S_{\sigma'}$ is a boundary stratum in $\Mbar_{0,N'}$,
  where $\sigma'$ is obtained from $\sigma$ by deleting the leg
  $\ell_N$ and stabilizing as in Section \ref{sec:StableCurves}. Let
  $v_0$ on $\sigma$ be the vertex with the leg $\ell_N.$

  \noindent \textit{Case 1. $\abs{v_0}\ge4.$}

  The $N'$-marked tree obtained by deleting $\ell_N$ is stable, and
  therefore is $\sigma'.$ The moduli dimension of
  $v_0$ is positive, and the corresponding vertex on $\sigma'$ has
  moduli dimension one less. All other vertices of $\sigma'$ have
  the same moduli dimension as the corresponding vertices of $\sigma,$ so $\dim
  S_{\sigma'}=\dim S_\sigma-1.$ Therefore
  $\mu_*([S_\sigma])=0\in\Lambda_{N'}^{\le\lambda_\sigma}$.

  \noindent \textit{Case 2. $\abs{v_0}=3$ and $\Delta_{v_0}$ consists
    of $\ell_N$, an edge $e_1$ connecting $v_0$ to $v_1$, and an edge
    $e_2$ connecting $v_0$ to $v_2$.}

  The stabilization $\sigma'$ is obtained from $\sigma$ by deleting
  $\ell_N,$ $v_0$, $e_1,$ and $e_2,$ and connecting $v_1$ and $v_2$
  with a new edge. The tree $\sigma'$ does not have the vertex $v_0$,
  which has moduli dimension zero, however all other vertices on
  $\sigma'$ have the same moduli dimension as the corresponding vertices on
  $\sigma.$ So $\dim S_{\sigma'}=\dim S_\sigma$ and
  $\lambda_{\sigma'}=\lambda_\sigma$. Therefore
  $\mu_*([S_\sigma])=[S_{\sigma'}]\in\Lambda_{N'}^{\le\lambda_\sigma}.$

  \noindent \textit{Case 3. $\abs{v_0}=3$ and $\Delta_{v_0}$ consists
    of $\ell_N$, another leg $\ell_i$, and an edge $e$ connecting
    $v_0$ to $v_1$.}

  The stabilization $\sigma'$ is obtained from $\sigma$ by deleting
  $\ell_N,$ $v_0$, and $e,$ and attaching the leg $\ell_i$ to
  $v_1$. The tree $\sigma'$ does not have the vertex $v_0$, which has
  moduli dimension zero, however all other vertices on $\sigma'$ have
  the same moduli dimension as the corresponding vertices on $\sigma.$
  So $\dim S_{\sigma'}=\dim S_\sigma$ and
  $\lambda_{\sigma'}=\lambda_\sigma$. Therefore
  $\mu_*([S_\sigma])=[S_{\sigma'}]\in\Lambda_{N'}^{\le\lambda_\sigma}.$
\end{proof}
Our main result in this section is Theorem
\ref{thm:PreservesFiltration}, showing that the filtration
$\{\Lambda_N^{\le\lambda}\}_\lambda$ is preserved by all Hurwitz
correspondences. Our proof uses the stratification of the moduli
spaces of admissible covers (Section
\ref{sec:StratificationOfAdmissibleCovers}).
\begin{thm}\label{thm:PreservesFiltration}
  Let $(\Gamma,\pi_\bB,\pi_\bA):\Mbar_{0,\bB}\ratt\Mbar_{0,\bA}$ be a
  Hurwitz correspondence. Then for every partition $\lambda$ of $k$,
  $[\Gamma]_*:H_{2k}(\Mbar_{0,\bB})\to H_{2k}(\Mbar_{0,\bA})$ sends
  $\Lambda_{\bB}^{\le\lambda}$ to $\Lambda_{\bA}^{\le\lambda}$.
\end{thm}
\begin{proof}[Proof of Theorem \ref{thm:PreservesFiltration}]
  By Lemma \ref{lem:ReduceToConnectedFullyMarked}, we may assume that
  $\Gamma$ is a union of connected components of a fully marked
  Hurwitz space $\H$, and by Lemma
  \ref{lem:ForgetfulMapPreservesFiltration}, we may assume that
  $\bA=\Afull$ as in Section \ref{sec:FullyMarked}.

  Let $\Hbar$ be the admissible covers compactification of $\H$ and
  let $\bar\Gamma$ be the closure of $\Gamma$ in $\Hbar$. The
  compactifications $\Hbar$ and $\bar\Gamma$ both have maps
  $\bar{\pi_{\bB}}$ and $\bar{\pi_{\bA}}$ to $\Mbar_{0,\bB}$ and
  $\Mbar_{0,\bA}.$ As in Section \ref{sec:CohomologyChow}, we have
  $[\Gamma]_*=(\bar{\pi_{\bA}})_*\circ(\bar{\pi_{\bB}})^*$.

  Fix $T_\tau$ any $k$-dimensional boundary stratum in
  $\Mbar_{0,\bB}$. Let $\lambda_\tau$ be the partition of $k$ induced
  by $T_\tau.$ We show
  $[\Gamma]_*([T_\tau])\in\Lambda_{\bA}^{\le\lambda_\tau}$. Since
  $\bar{\pi_{\bB}}$ is flat,
  \begin{align*}
    (\bar{\pi_\bB})^*([T_\tau])=\sum_Jm_J[J],
  \end{align*}
  where the sum is over irreducible components $J$ of the preimage
  $(\bar{\pi_{\bB}})^{-1}(T_\tau)$ in $\bar\Gamma$, and $m_J$ is a
  positive integer multiplicity.

  To show that $(\bar{\pi_\bA})_*\circ(\bar{\pi_{\bB}})^*([T_\tau])$
  is in $\Lambda_\bA^{\le\lambda_\tau},$ it suffices to show that for
  $J$ an irreducible component of $(\bar{\pi_\bB})^{-1}(T_\tau),$ the
  pushforward $(\bar{\pi_\bA})_*([J])$ is in
  $\Lambda_\bA^{\le\lambda_\tau}.$

  Fix such a $J$ --- it is an irreducible component of a boundary
  stratum $G_\gamma$ of $\Hbar$, with combinatorial type
  $\gamma=(\sigma,\tau,d^{\vert},f^{\vert},(F_v)_{v\in\mathcal{V}(\sigma)},(\br_v)_{v\in\mathcal{V}(\sigma)},(\rm_v)_{v\in\mathcal{V}(\sigma)})$
  (Section \ref{sec:StratificationOfAdmissibleCovers}). There is a
  decomposition $G_\gamma=\prod_{w\in\mathcal{V}(\tau)}\Hbar_w,$ where
  $\Hbar_w$ is an admissible covers space of dimension $\md(w),$ and a
  decomposition
  \begin{align*}
    J=\prod_{w\in\mathcal{V}(t)}\bar{\mathcal{J}_w},
  \end{align*}
  where $\bar{\mathcal{J}_w}$ is an irreducible component of
  $\Hbar_w.$ Each factor $\bar{\mathcal{J}_w}$ admits a map
  $\bar{\pi_w^{\mathrm{source}}}$ to
  $\prod_{v\in(f^{\vert})^{-1}(w)}\Mbar_{0,\Delta_v}$ and a finite
  flat map $\bar{\pi_w^{\mathrm{target}}}$ to $\Mbar_{0,\Delta_w}.$ We
  conclude that
  $\dim\bar{\mathcal{J}_w}=\dim\Mbar_{0,\Delta_w}=\md(w).$ The key
  observation of our proof is that this decomposition of $J$ also
  induces the partition $\lambda_\tau.$

  The cohomology (respectively Chow) ring of
  $\prod_{v\in(f^{\vert})^{-1}(w)}\Mbar_{0,\Delta_v}$ is the tensor
  product of the cohomology (respectively Chow) rings of the factors,
  which in turn are generated by the classes of boundary strata. Thus
  we may write
  \begin{align*}
    (\bar{\pi_w^{\mathrm{source}}})_*([\bar{\mathcal{J}_w}])=\sum_{q\in Q_w}\beta_q\bigotimes_{v\in(f^{\vert})^{-1}(w)}[S_{wv}^q],
  \end{align*}
  where $Q_w$ is a finite set, $\beta_q$ is an integer multiplicity,
  and $S_{wv}^q$ is a boundary stratum of dimension $k_{wv}^q$ in
  $\Mbar_{0,\Delta_v}.$ Denote by $\lambda_{wv}^q$ the partition of
  $k_{wv}^q$ induced by $S_{wv}^q$. For every $q\in Q_w,$
  \begin{align*}
    \sum_{v\in(f^{\vert})^{-1}(w)}k_{wv}^q=\dim\bar{\mathcal{J}_w}=\md(w),
  \end{align*}
  so $\bigcup_{v\in(f^{\vert})^{-1}(w)}\lambda_{wv}^q$ is a partition
  of $\md(w).$ The map $\bar{\pi_{\bA}}|_J$ decomposes as a product:
  \begin{align*}
    J=\prod_{w\in\mathcal{V}(\tau)}\bar{\mathcal{J}_w}\xrightarrow{\quad\prod_{w\in\mathcal{V}(\tau)}\bar{\pi_w^{\mathrm{source}}}\quad}\prod_{w\in\mathcal{V}(\tau)}\prod_{v\in(f^{\vert})^{-1}(w)}\Mbar_{0,\Delta_v}\cong
    S_\sigma\xrightarrow{\quad\iota\quad}\Mbar_{0,\bA},
  \end{align*}
  so
  \begin{align*}
    \left(\bar{\pi_\bA}|_J\right)_*([J])&=\iota_*\left(\bigotimes_{w\in\mathcal{V}(\tau)}(\bar{\pi_w^{\mathrm{source}}})_*([\bar{\mathcal{J}_w}]\right)\\
    &=\iota_*\left(\bigotimes_{w\in\mathcal{V}(\tau)}\left(\sum_{\quad
          q\in
          Q_w\quad}\beta_q\bigotimes_{v\in(f^{\vert})^{-1}(w)}[S_{wv}^q]\right)\right)\\
    &=\iota_*\left(\sum_{\quad(q(w))_w\in\prod_{w\in\mathcal{V}(\tau)}Q_w\quad}\left(\prod_{w\in\mathcal{V}(\tau)}\beta_{q(w)}\right)\bigotimes_{\quad
        w\in\mathcal{V}(\tau)\quad}\bigotimes_{v\in(f^{\vert})^{-1}(w)}[S_{wv}^{q(w)}]\right)\\
    &=\sum_{(q(w))_w\in\prod_{w\in\mathcal{V}(\tau)}Q_w}\left(\prod_{w\in\mathcal{V}(\tau)}\beta_{q(w)}\right)\left[\iota\left(\prod_{\quad
          w\in\mathcal{V}(\tau)\quad}\prod_{v\in(f^{\vert})^{-1}(w)}S_{wv}^{q(w)}\right)\right].
  \end{align*}
  For fixed $(q(w))_w$, the image
  $$\iota\left(\prod_{\quad
      w\in\mathcal{V}(\tau)\quad}\prod_{v\in(f^{\vert})^{-1}(w)}S_{wv}^{q(w)}\right)$$
  under the gluing morphism is a $k$-dimensional boundary stratum in
  $\Mbar_{0,\bA}$. It induces partition of $k$:
  $$\bigcup_{\quad w\in\mathcal{V}(\tau)\quad}\bigcup_{v\in(f^{\vert})^{-1}(w)}\lambda_{wv}^{q(w)}.$$
  As expressed this is clearly a refinement
  of
  $$\lambda_\tau=(\md(w))_{\text{$w\in\mathcal{V}(\tau)$ with
      $\md(w)\ne0$}}.$$
  Thus $(\bar{\pi_\bA})_*([J])$ is in $\Lambda_\bA^{\le\lambda_\tau}.$
\end{proof}

\section{Dynamical degrees of Hurwitz correspondences}\label{sec:KComposability}
Theorem \ref{thm:PreservesFiltration} implies that
$[\Gamma]_*:H_{2k}(\Mbar_{0,\bB})\to H_{2k}(\Mbar_{0,\bA})$ can be written,
in multiple different ways, as a block lower triangular matrix. In the
case of a Hurwitz self-correspondence, one can ask which of these
blocks contains the dynamical degree --- the dominant eigenvalue. This
section addresses this question.

Set $$\Lambda_N^{<(k)}:=\sum_{\text{$\lambda$ has $\ge2$
    parts}}\Lambda_N^{\le\lambda}$$
and $$\Omega_N^k:=\frac{H_{2k}(\Mbar_{0,N})}{\Lambda_N^{<(k)}}.$$ We
prove:

\medskip

\noindent\textbf{Theorem \ref{thm:DynamicalDegreeInTopBlock}.} \textit{Let $\Gamma:\Mbar_{0,N}\ratt\Mbar_{0,N}$ be a dominant Hurwitz
  correspondence. Then the $k$th dynamical degree of $\Gamma$ is the
  absolute value of the dominant eigenvalue of
  $[\Gamma]_*:\Omega_N^k\to\Omega_N^k.$}

\subsection{Hurwitz correspondences on alternate compactifications of
  $\M_{0,N}$}\label{sec:ProofOfProp}
\begin{Def}
  A weight datum $\be=(\epsilon_1,\ldots,\epsilon_N)$ is called
  \newword{minimal} if, for every subset $P\subseteq\{1,\ldots,N\}$,
  $\sum_{i\in P}\epsilon_i>1$ if and only if $\sum_{i\not\in
    P}\epsilon_i<1$.
\end{Def}
\begin{lem}\label{lem:UniqueVertex}
  Let $\be=(\epsilon_1,\ldots,\epsilon_N)$ be a minimal weight
  datum. Then every $N$-marked stable tree $\sigma$ has a unique
  vertex that is $\be$-stable.
\end{lem}
\begin{proof}
  Every tree has at least one $\be$-stable vertex (Section
  \ref{sec:WeightedStableCurves}). Conversely, let $\sigma$ be an
  $N$-marked stable tree. If $\sigma$ has a unique vertex, we are
  done. If not, any edge $e$ disconnects $\sigma$ into two components
  $\sigma_1$ and $\sigma_2.$ Since $\be$ is minimal, exactly one of
  $\sum_{\text{$\ell_i$ on $\sigma_1$}}\epsilon_i$ and
  $\sum_{\text{$\ell_i$ on $\sigma_2$}}\epsilon_i$ is greater than
  1. If $\sum_{\text{$\ell_i$ on $\sigma_1$}}\epsilon_i>1$, then no
  vertex on $\sigma_2$ is $\be$-stable, and if $\sum_{\text{$\ell_i$
      on $\sigma_2$}}\epsilon_i>1,$ then no vertex on $\sigma_1$ is
  $\be$-stable. Since $e$ was arbitrary, $\sigma$ has at most one
  $\be$-stable vertex.
\end{proof}
\begin{prop}\label{prop:MinimalKernel}
  Let $\be$ be a minimal weight datum, and let
  $\rho_{\be}:\Mbar_{0,N}\to\Mbar_{0,N}(\be)$ be the reduction
  morphism. Then $\ker((\rho_{\be})_*)\subseteq H_{2k}(\Mbar_{0,N})$
  contains $\Lambda_N^{<(k)}.$
\end{prop}
\begin{proof}
  Let $S_\sigma$ be any boundary stratum in $\Lambda_N^{<(k)}$. The
  partition $\lambda_\sigma$ has at least two parts, so $\sigma$ has
  at least two vertices with positive moduli dimension. By Lemma
  \ref{lem:UniqueVertex}, at least one of these is not $\be$-stable,
  and so by Lemma \ref{lem:CriterionForKernel},
  $S_\sigma\in\ker((\rho_{\be})_*).$
\end{proof}
There are many minimal weight data, giving rise to non-isomorphic
spaces of weighted stable curves, as follows.
\begin{ex}
  $\epsilon_1=1+\frac{N}{10^N}$, and
  $\epsilon_i=\frac{1}{N-1}-\frac{1}{10^{N}}$ for
  $i=2,\ldots,N$. Then $\Mbar_{0,N}(\be)\cong\P^{N-3}$
\end{ex}
\begin{ex}\label{ex:EpsilonDagger}
  Let $\be^{\dagger}$ be the following minimal weight datum:
  \begin{itemize}
  \item \textbf{$N$ odd.}  $\epsilon_i=\frac{2}{N}+\frac{1}{10^N}$ for
    all $i$.
  \item\textbf{$N$ even.} $\epsilon_1=\frac{2}{N}+\frac{1}{10^N}$, and
    $\epsilon_2=\cdots=\epsilon_N=\frac{2}{N}-\frac{1}{N\cdot10^N}.$
  \end{itemize}
  The subsets of $\{\epsilon_1,\ldots,\epsilon_N\}$ with sum greater
  than 1 are those with size more than $N/2$, or those with size $N/2$
  containing $\epsilon_1.$ Set
  $X_N^{\dagger}=\Mbar_{0,N}(\be^\dagger)$ and
  $\rho^\dagger:\Mbar_{0,N}\to X_N^\dagger$ the reduction
  morphism. $X_N^\dagger$ has Picard rank $N;$ in particular, it is
  not isomorphic to $\P^{N-3}$.

  There is a simple description of an $\be^\dagger$-stable curve: $C$
  an irreducible genus zero curve, and $p_1,\ldots,p_N$ marked points,
  not necessarily distinct, such that
  \begin{itemize}
  \item at least 3 distinct points on $C$ are marked,
  \item no point on $C$ has greater than $N/2$ marks, and
  \item if a point on C has exactly $N/2$ marks, then $p_1$ is not
    among them.
  \end{itemize}
\end{ex}

By Proposition \ref{prop:MinimalKernel} and Lemma
\ref{lem:DynamicalDegreeInQuotient}, we obtain:
\begin{thm}\label{thm:DynamicalDegreeInTopBlock}
  Let $\Gamma:\Mbar_{0,N}\ratt\Mbar_{0,N}$ be a dominant Hurwitz
  correspondence. Then the $k$th dynamical degree of $\Gamma$ is the
  absolute value of the dominant eigenvalue of
  $[\Gamma]_*:\Omega_N^k\to\Omega_N^k.$
\end{thm}
Thus the dynamical degrees of $\Gamma$ may be computed without any
information about its action on the subspace $\Lambda_N^{<(k)}$. This
suggests that this action is irrelevant to the dynamics of $\Gamma$ on
$\M_{0,N}.$ Relatedly, we have
\begin{prop}
  Any effective cycle $\alpha\in\Lambda_N^{<(k)}$ is supported on the
  boundary of $\Mbar_{0,N}$.
\end{prop}
\begin{proof}
  By Proposition \ref{prop:MinimalKernel}, $\alpha$ is in the kernel
  of $\rho^\dagger_*:\Mbar_{0,N}\to X_N^\dagger$, where $X_N^\dagger$
  is as in Example \ref{ex:EpsilonDagger}. The interior $\M_{0,N}$
  maps isomorphically onto an open set on $X_N^\dagger.$ Thus the
  support of $\alpha$ does not intersect $\M_{0,N}.$
\end{proof}
\begin{cor}\label{cor:ContainedInBoundary}
  Let $\Gamma:\Mbar_{0,N}\ratt\Mbar_{0,N}$ be a Hurwitz correspondence
  and $S_\sigma\in\Lambda_N^{<(k)}$ be a boundary stratum. Then the
  cycle $[\Gamma]_*([S_\sigma])$ is supported on the boundary of
  $\Mbar_{0,N}$.
\end{cor}
Interestingly, for $k>\frac{\dim\M_{0,N}}{2}$, we can contract these
strata and keep $k$-stability of $\Gamma$ (Corollary
\ref{cor:KStabilityOnXNDagger}).
\begin{lem}\label{lem:KernelGeneratedByBoundaryStrata}
  Let $\Mbar_{0,N}(\be)$ be a moduli space of
  weighted stable curves, with reduction morphism
  $\rho_{\be}:\Mbar_{0,N}\to\Mbar_{0,N}(\be)$. Then
  the kernel of
  $(\rho_{\be})_*:H_{2k}(\Mbar_{0,N})\to
  H_{2k}(\Mbar_{0,N}(\be))$ is generated by classes
  of $k$-dimensional boundary strata.
\end{lem}
\begin{proof}
  The homology groups of $\Mbar_{0,N}$ and $\Mbar_{0,N}(\be)$ are
  generated by the classes of boundary strata. Denote by $\mathcal{F}$
  the free $\R$-vector space on the set of $k$-dimensional boundary
  strata in $\Mbar_{0,N}$, and denote by $\mathcal{F}^{\be}$ the free
  $\R$-vector space on the set of $k$-dimensional boundary strata in
  $\Mbar_{0,N}(\be).$ Let $\mathcal{R}\subseteq\mathcal{F}$ and
  $\mathcal{R}^{\be}\subseteq\mathcal{F}^{\be}$ be the subspaces of
  relations, i.e. linear combinations of boundary strata that are
  homologous to zero in $\Mbar_{0,N}$ and $\Mbar_{0,N}(\be)$
  respectively. Define a map
  $\tilde{\rho_{\be}}:\mathcal{F}\to\mathcal{F}^{\be}$ on generators
  as follows. If $\dim(\rho_{\be}(S_\sigma))<k,$ set
  $\tilde{\rho_{\be}}(S_\sigma)=0,$ else set
  $\tilde{\rho_{\be}}(S_\sigma)=\rho_{\be}(S_\sigma).$ We have
  \begin{center}
    \begin{tikzpicture}
      \matrix(m)[matrix of math nodes,row sep=3em,column
      sep=4em,minimum width=2em] {
        0&\mathcal{R}&\mathcal{F}&H_{2k}(\Mbar_{0,N})&0\\
        0&\mathcal{R}^{\be}&\mathcal{F}^{\be}&H_{2k}(\Mbar_{0,N}(\be))&0\\};
      \path[-stealth] (m-1-1) edge (m-1-2);  
      \path[-stealth] (m-1-2) edge (m-1-3);  
      \path[-stealth] (m-1-3) edge (m-1-4);  
      \path[-stealth] (m-1-4) edge (m-1-5);  
      \path[-stealth] (m-2-1) edge (m-2-2);  
      \path[-stealth] (m-2-2) edge (m-2-3);  
      \path[-stealth] (m-2-3) edge (m-2-4);  
      \path[-stealth] (m-2-4) edge (m-2-5);  
      \path[-stealth] (m-1-2) edge node [left] {$\tilde{\rho_{\be}}$} (m-2-2); 
      \path[-stealth] (m-1-3) edge node [left] {$\tilde{\rho_{\be}}$} (m-2-3); 
      \path[-stealth] (m-1-4) edge node [left] {$(\rho_{\be})_*$} (m-2-4); 
    \end{tikzpicture}
  \end{center}
  Ceyhan \cite{Ceyhan2009} gives generators for $\cR^{\be}.$ Each
  generator is the image under $\tilde{\rho_{\be}}$ of some generator
  $\cR(\sigma,v,i_1,\ldots,i_4)$ in $\cR$ (see Section
  \ref{sec:StableCurves}). Thus $\tilde{\rho_{\be}}(\cR)=\cR^{\be},$
  and the kernel of $\cF\to H_{2k}(\Mbar_{0,N}(\be))$ is the sum
  $\cR+\ker(\tilde{\rho_{\be}}).$

  On the other hand, any two $k$-dimensional boundary strata in
  $\Mbar_{0,N}$ with the same $k$-dimensional image in
  $\Mbar_{0,N}(\be)$ are homologous. Thus
  $\ker(\tilde{\rho_{\be}})/(\ker(\tilde{\rho_{\be}})\cap\cR)$ is
  generated by boundary strata, so $\ker((\rho_{\be})_*)$ is also
  generated by boundary strata.
\end{proof}
\begin{prop}\label{prop:WorksForHalfK}
  Let $\be^\dagger,$ $X_N^\dagger,$ and $\rho^\dagger$ be as in
  Example \ref{ex:EpsilonDagger}. Then for
  $k\ge\frac{\dim\M_{0,N}}{2},$ $$\Lambda_N^{<(k)}=\ker(\rho^\dagger_*)\subseteq
  H_{2k}(\Mbar_{0,N}).$$
\end{prop}
\begin{proof}
  Fix $k\ge\frac{\Mbar_{0,N}}{2}.$ By Proposition
  \ref{prop:MinimalKernel}, we have
  $\Lambda_N^{<(k)}\subseteq\ker(\rho^\dagger_*).$ Suppose $S_\sigma$
  is a $k$-dimensional boundary stratum not in $\Lambda_N^{<(k)}.$
  Then $\lambda_\sigma=(k)$ and $\sigma$ has a unique vertex $v$ with
  positive moduli dimension $\md(v)=k.$ So $v$ has valence
  \begin{align*}
    k+3\ge\frac{\dim\M_{0,N}}{2}+3=\frac{N+3}{2}>\frac{N}{2}+1.
  \end{align*}
  Thus, for every flag $\delta\in\Delta_v,$ the set
  $\{i|\delta=\delta(v\to i)\}$ has size less than $N/2,$ so
  $\sum_{i|\delta=\delta(v\to i)}\epsilon_i<1.$ We conclude that
  \begin{align*}
    \sum_{\delta\in\Delta_v}\min\left\{1,\sum_{i|\delta=\delta(v\to i)}\epsilon_i\right\}=\sum_{i=1}^N\epsilon_i>2,
  \end{align*}
  so $v$ is $\be$-stable. By Lemma \ref{lem:CriterionForKernel},
  $S_\sigma\not\in\ker(\rho^\dagger_*).$

  Thus a $k$-dimensional boundary stratum is in $\ker(\rho^\dagger_*)$
  exactly if it is in $\Lambda_N^{<(k)}.$ By Lemma
  \ref{lem:KernelGeneratedByBoundaryStrata},
  $\ker(\rho^\dagger_*)=\Lambda_N^{<(k)}.$
\end{proof}
By Proposition \ref{prop:WorksForHalfK} and Lemma
\ref{lem:CriterionForComposability}:
\begin{cor}\label{cor:KStabilityOnXNDagger}
  Let $\Gamma:X_N^{\dagger}\ratt X_N^{\dagger}$ be a dominant Hurwitz
  correspondence. Then for $k\ge\frac{\dim\M_{0,N}}{2}$:
  \begin{enumerate}[(i)]
  \item We have a commutative diagram\label{item:CommDiag}
    \begin{center}
      \begin{tikzpicture}
        \matrix(m)[matrix of math nodes,row sep=3em,column
        sep=4em,minimum width=2em] {
          \Omega_{N}^k&\Omega_{N}^k\\
          H_{2k}(X_{N}^\dagger)&H_{2k}(X_{N}^\dagger)\\
        }; \path[-stealth] (m-1-1) edge node [above] {$[\Gamma]_*$}
        (m-1-2); \path[-stealth] (m-1-1) edge node [left]
        {$(\rho^\dagger)_*$} node [right] {$\cong$} (m-2-1);
        \path[-stealth] (m-1-2) edge node [right]
        {$(\rho^\dagger)_*$} node [left] {$\cong$} (m-2-2);
        \path[-stealth] (m-2-1) edge node [above] {$[\Gamma]_*$}
        (m-2-2);
      \end{tikzpicture}
    \end{center}
  \item $[\Gamma]_*:H_{2k}(X_N^\dagger)\to H_{2k}(X_N^{\dagger})$
    preserves the cone of effective classes,
    and\label{item:PreservesCone}
  \item $\Gamma$ is $k$-stable on
    $X_N^{\dagger}.$\label{item:KKStable}
  \end{enumerate}
\end{cor}
By Theorem \ref{thm:DynamicalDegreeInTopBlock}, the dynamical degree
of a Hurwitz correspondence $\Gamma:\Mbar_{0,N}\ratt\Mbar_{0,N}$ is
the absolute value of the dominant eigenvalue of the induced action of
$\Gamma$ on the quotient vector space $\Omega_N^k$. By Corollary
\ref{cor:KStabilityOnXNDagger}, for $k\ge\frac{\dim\M_{0,N}}{2}$, this action
on $\Omega_N^k$ also has an interpretation as the $k$-stable action of
$\Gamma$ on $H_{2k}(X_N^\dagger).$ We obtain as a corollary:
\begin{cor}
  For $k\ge\frac{\dim\M_{0,N}}{2},$
  $[\Gamma]_*:\Omega_N^k\to\Omega_N^k$ has a nonnegative dominant eigenvalue.
\end{cor}
\begin{proof}
  By Corollary \ref{cor:KStabilityOnXNDagger}
  \eqref{item:PreservesCone}, $[\Gamma]_*:H_{2k}(X_N^\dagger)\to
  H_{2k}(X_N^\dagger)$ preserves the cone of effective classes, thus
  as in Remark \ref{rem:EffectiveCone} has a nonnegative dominant eigenvalue.
\end{proof}
For $k<\frac{\dim\M_{0,N}}{2},$ we ask
\begin{enumerate}[(1)]
\item Can we interpret $[\Gamma]_*:\Omega_N^k\to\Omega_N^k$ as the
  $k$-stable homological action of $\Gamma$ on an alternate smooth
  compactification of $\M_{0,N}?$\label{Question1}
\item Equivalently, is there an alternate smooth compactification
  $X_N$ of $\M_{0,N}$ admitting a birational morphism
  $\rho:\Mbar_{0,N}\to X_N$ such that the kernel of
  $\rho_*:H_{2k}(\Mbar_{0,N})\to H_{2k}(\Mbar_{0,N})$ is
  $\Lambda_N^{<(k)}?$\label{Question2}
\item Is there an alternate smooth compactification $X_N$ of
  $\M_{0,N}$ such that every Hurwitz correspondence $\Gamma:X_N\ratt
  X_N$ is $k$-stable?\label{Question3}
\end{enumerate}
Smyth (\cite{Smyth2009}) defined a \emph{modular compactification} of
$\M_{0,N}$ to be one that extends its moduli space
interpretation. (These include $\Mbar_{0,N}$ and any space of weighted
stable curves.) If the discussion is restricted to modular
compactifications, the answers to questions \ref{Question1} and
\ref{Question2} are both ``no.'' In fact:
\begin{prop}[\cite{RamadasThesis}]\label{prop:DoesNotWorkForHalfK}
  Fix $N,$ $k$ with $1\le k<\frac{\dim\M_{0,N}}{2},$ and a nonempty
  subset $L\subseteq\{\mbox{partitions of $k$}\}$. Then there is no
  smooth or projective modular compactification $X_N$ of $\M_{0,N}$
  with reduction morphism $\rho:\Mbar_{0,N}\to X_N$ such that the
  kernel of $(\rho)_*:H_{2k}(\Mbar_{0,N})\to H_{2k}(X_N)$ is equal to
  $\sum_{\lambda\in L}\Lambda_N^{\le\lambda}.$
\end{prop}
There exist other natural classes of compactifications of
$\M_{0,N}$ (e.g. \cite{GiansiracusaJensenMoon2013},
\cite{CartwrightMaclagan2016}). However, any compactification arising
in the construction of \cite{GiansiracusaJensenMoon2013} or
\cite{CartwrightMaclagan2016} that is \emph{not} modular in the sense
of Smyth is singular, thus may not be used to study the dynamics of
Hurwitz correspondences. This leads us to conjecture that the answers
to all three questions are in the negative.

\subsection{The filtration $\{\Lambda_N^{\le\lambda}\}$ and quotients
  $\Omega_N^k$}\label{sec:DiscussionFiltration} In light of Theorems
\ref{thm:PreservesFiltration} and \ref{thm:DynamicalDegreeInTopBlock},
we ask: What are the dimensions of the subspaces
$\Lambda_N^{\le\lambda}$ and quotients $\Omega_N^k?$ In particular,
the $k$th dynamical degree of a Hurwitz correspondence $\Gamma$,
namely the dominant eigenvalue of $[\Gamma]_*:H_{2k}(\Mbar_{0,N})\to
H_{2k}(\Mbar_{0,N})$, is in fact the dominant eigenvalue of the
smaller matrix $[\Gamma]_*:\Omega_N^k\to\Omega_N^k$. What are the
comparative sizes of the two matrices --- how much easier does Theorem
\ref{thm:DynamicalDegreeInTopBlock} make computation of dynamical
degrees?

There is a recursive formula (\cite{Keel1992}) for the Betti numbers
of $\Mbar_{0,N}.$ Moon \cite[Corollary 5.3.2]{Moon2011} gives a
formula for the Poincar\'e polynomial of $X_N^\dagger$. For
$k\ge\frac{\dim\M_{0,N}}{2}$, the quotient $\Omega_N^k$ is isomorphic
to $H_{2k}(X_N^\dagger),$ so this tells us $\dim\Omega_N^k$ for $k$ in
this range. This allows us, in principle, to compare
$H_{2k}(\Mbar_{0,N})$ and $\Omega_N^k$ for
$k\ge\frac{\dim\M_{0,N}}{2}$.

For $k=\dim\M_{0,N}-1$ (divisors) and $k=1$ (curve classes), we can be
more explicit:

\medskip

\noindent\textbf{The case $k=\dim\M_{0,N}-1$.} The homology group $H_{2k}(\Mbar_{0,N})=\Pic(\Mbar_{0,N})$ has
dimension $\frac{2^{N}-N^2+N-2}{2}$ and is generated by the classes of
boundary divisors.  $\Mbar_{0,N}$ carries $N$ tautological line
bundles:
\begin{Def}
  For $i\in\{1,\ldots,N\}$, define a line bundle $\cL_i$ on
  $\Mbar_{0,N}$ as follows: the fiber over $[C,p_1,\ldots,p_N]$ is
  $T_{p_i}^{\vee}C,$ that is, the cotangent line to $C$ at the marked
  point $p_i$.
\end{Def}
By \cite{FarkasGibney2003}, we have
\begin{enumerate}[(i)]
\item The boundary divisors $\{S_\sigma|\mbox{$\lambda_\sigma$ has
    exactly two parts}\}$ are a basis for $\Lambda_N^{<(k)}$. Thus
  $$\dim\Lambda_N^{<(k)}=\frac{1}{2}\sum_{j=3}^{N-3}\binom{N}{j}=\frac{1}{2}(2^N-2-2N-N(N-1)).$$
\item The divisor classes $\{c_1(\cL_i)\}_{i=1,\ldots,N}$ are a basis
  for $\Omega_N^k.$ Thus $\dim\Omega_N^k=N.$
\end{enumerate}
It follows from Theorem \ref{thm:DynamicalDegreeInTopBlock} that the
$(N-4)$th dynamical degree of $\Gamma$ is an algebraic integer of
degree at most $N$. This generalizes a result in
\cite{KochRoeder2015}.

\medskip

\noindent \textbf{The case $k=1$.} By Poincar\'e duality,
$$\dim H_2(\Mbar_{0,N})=\dim
H_{2(\dim\M_{0,N}-1)}(\Mbar_{0,N})=\frac{2^{N}-N^2+N-2}{2}.$$
Since there is only one partition of 1, the filtration
$\{\Lambda_N^{\le\lambda}\}$ of $H_2(\Mbar_{0,N})$ is trivial, and
$H_2(\Mbar_{0,N})=\Omega_N^1$.

By Corollary \ref{cor:FiltrationNontrivial}, for $k$ strictly between
1 and $\dim\M_{0,N}$, $\dim\Omega_N^k<\dim H_{2k}(\Mbar_{0,N})$.

\section*{References}
\bibliographystyle{plain}
\bibliography{../HurwitzRefs}

\begin{thebibliography}{10}

\bibitem{Bedford2011}
Eric Bedford.
\newblock {The dynamical degrees of a mapping}.
\newblock {\em ArXiv e-prints}, 2011.
\newblock \href{http://arxiv.org/abs/1110.1741}{\texttt{arXiv:1110.1741}}.

\bibitem{BoucksomFavreJonsson2009}
S{\'e}bastien Boucksom, Charles Favre, and Mattias Jonsson.
\newblock Differentiability of volumes of divisors and a problem of {T}eissier.
\newblock {\em Journal of Algebraic Geometry}, 18:279--308, 2009.

\bibitem{Guedj2010}
Serge Cantat, Antoine Chambert-Loir, and Vincent Guedj.
\newblock {\em Quelques Aspects des Systemes Dynamiques Polynomiaux}.
\newblock Soci{\'e}t{\'e} math{\'e}matique de France, 2010.

\bibitem{CartwrightMaclagan2016}
Dustin Cartwright and Diane Maclagan.
\newblock Tree compactifications.
\newblock In preparation.

\bibitem{Ceyhan2009}
{\"O}zg{\"u}r Ceyhan.
\newblock Chow groups of the moduli spaces of weighted pointed stable curves of
  genus zero.
\newblock {\em Advances in Mathematics}, 221:1964--1978, 2009.

\bibitem{DinhSibony2005}
Tien-Cuong Dinh and Nessim Sibony.
\newblock Une borne sup\'erieure pour l'entropie topologique d'une application
  rationelle.
\newblock {\em Annals of Mathematics}, 161(3):1637--1644, 2005.

\bibitem{DinhSibony2008}
Tien-Cuong Dinh and Nessim Sibony.
\newblock Upper bound for the topological entropy of a meromorphic
  correspondence.
\newblock {\em Israel Journal of Mathematics}, 163(1):29--44, 2008.

\bibitem{DouadyHubbard1993}
Adrien Douady and John~H. Hubbard.
\newblock A proof of {T}hurston's topological characterization of rational
  functions.
\newblock {\em Acta Mathematica}, 171(2):263--297, 1993.

\bibitem{FarkasGibney2003}
Gavril Farkas and Angela Gibney.
\newblock The {M}ori cones of moduli spaces of pointed curves of small genus.
\newblock {\em Transactions of the American Mathematical Society},
  355:1183--1199, 2003.

\bibitem{Friedland1991}
Shmuel Friedland.
\newblock Entropy of polynomial and rational maps.
\newblock {\em Annals of Mathematics}, 133(2):359--368, 1991.

\bibitem{FulgerLehmann2014}
Mihai Fulger and Brian Lehmann.
\newblock {Positive cones of dual cycle classes}.
\newblock {\em ArXiv e-prints}, 2014.
\newblock \href{http://arxiv.org/abs/1408.5154}{\texttt{arXiv:1408.5154}}.

\bibitem{Fulton1998}
William Fulton.
\newblock {\em Intersection Theory}.
\newblock Springer-Verlag New York, second edition, 1998.

\bibitem{GiansiracusaJensenMoon2013}
Noah Giansiracusa, David Jensen, and Han-Bom Moon.
\newblock {GIT} compactifications of {$\mathcal{M}_{0,N}$} and flips.
\newblock {\em Advances in Mathematics}, 248:242--278, 2013.

\bibitem{Gromov2003}
Mikha{\"i}l Gromov.
\newblock On the entropy of holomorphic maps.
\newblock {\em L'Enseignement Math\'ematique}, 49:217--235, 2003.

\bibitem{HarrisMorrison1998}
Joe Harris and Ian Morrison.
\newblock {\em Moduli of Curves}.
\newblock Springer New York, 1998.

\bibitem{HarrisMumford1982}
Joe Harris and David Mumford.
\newblock On the {K}odaira dimension of the moduli space of curves.
\newblock {\em Inventiones {M}athematicae}, 67:23--86, 1982.

\bibitem{Hassett2003}
Brendan Hassett.
\newblock Moduli spaces of weighted pointed stable curves.
\newblock {\em Advances in Mathematics}, 173(2):316--352, 2003.

\bibitem{Keel1992}
Sean Keel.
\newblock Intersection theory on the moduli space of stable $n$-pointed curves
  of genus zero.
\newblock {\em Transactions of the American Mathematical Society}, 330(2),
  1992.

\bibitem{Koch2013}
Sarah Koch.
\newblock Teichm{\"u}ller theory and critically finite endomorphisms.
\newblock {\em Advances in Mathematics}, 248:573--617, 2013.

\bibitem{KochRoeder2015}
Sarah Koch and Roland K.~W. Roeder.
\newblock Computing dynamical degrees of rational maps on moduli space.
\newblock {\em Ergodic Theory and Dynamical Systems}, FirstView:1--42, 2015.

\bibitem{KockVainsencher}
Joachim Kock and Israel Vainsencher.
\newblock {\em An Invitation to Quantum Cohomology: Kontsevich's Formula for
  Rational Plane Curves}.
\newblock Birkh{\"a}user, 2006.

\bibitem{KontsevichManin1994}
Maxim Kontsevich and Yuri Manin.
\newblock Gromov-{W}itten classes, quantum cohomology, and enumerative
  geometry.
\newblock {\em Comm. Math. Phys.}, 164(3):525--562, 1994.

\bibitem{Moon2011}
Han-Bom Moon.
\newblock {\em {Birational geometry of moduli spaces of curves of genus zero}}.
\newblock PhD thesis, Seoul National University, 2011.

\bibitem{RamadasThesis}
Rohini Ramadas.
\newblock {\em Dynamics of {H}urwitz correspondences}.
\newblock PhD thesis, University of Michigan.
\newblock In preparation.

\bibitem{Ramadas2016}
Rohini Ramadas.
\newblock {Dynamical degrees of {H}urwitz correspondences}.
\newblock {\em ArXiv e-prints}, 2016.
\newblock \href{http://arxiv.org/abs/1602.02846}{\texttt{arXiv:1602.02846}}.

\bibitem{Roeder2013}
Roland~K.W. Roeder.
\newblock The action on cohomology by compositions of rational maps.
\newblock {\em Mathematical Research Letters}, 22(2), 2013.

\bibitem{RomagnyWewers2006}
Matthieu Romagny and Stefan Wewers.
\newblock {H}urwitz spaces.
\newblock {\em S{\'e}minaires et Congr{\`e}s}, 13:313--341, 2006.

\bibitem{RussakovskiiShiffman1997}
Alexander Russakovskii and Bernard Shiffman.
\newblock Value distribution for sequences of rational mappings and complex
  dynamics.
\newblock {\em Indiana University Mathematics Journal}, 46(3):897--932, 1997.

\bibitem{SchneiderTam2007}
Hans Schneider and Bit-Shun Tam.
\newblock Matrices leaving a cone invariant.
\newblock In Leslie Hogben, editor, {\em Handbook of Linear Algebra}, pages
  26--1---26--17. Chapman \& Hall/CRC, 2007.

\bibitem{Smyth2009}
David~Ishii Smyth.
\newblock Towards a classification of modular compactifications of
  ${\cM}_{g,n}$.
\newblock {\em Inventiones {M}athematicae}, 192(2):459--503, 2009.

\bibitem{Truong2015}
Tuyen~Trung {Truong}.
\newblock {(Relative) dynamical degrees of rational maps over an algebraic
  closed field}.
\newblock {\em ArXiv e-prints}, 2015.
\newblock \href{http://arxiv.org/abs/1501.01523}{\texttt{arXiv:1501.01523}}.

\end{thebibliography}

\end{document}